\documentclass[11pt,reqno]{amsart}
\usepackage{amssymb, amsmath, amsthm,amsrefs,marvosym,wasysym}
\usepackage{latexsym}
\usepackage{color}
\usepackage{exscale}
\usepackage{fullpage}
\usepackage{rotating}
\usepackage{bm, bbm, mathrsfs}
\usepackage{graphics, graphicx}
\newtheorem{theorem}{Theorem}[section]
\newtheorem{lemma}[theorem]{Lemma}
\newtheorem{proposition}[theorem]{Proposition}

\newtheorem{remark}{Remark}[section]

\def\eq#1{(\ref{#1})}
\def\nn{\nonumber}
\def\({\left(\begin{array}{cccccc}}
\def\){\end{array}\right)}

\def\eq#1{(\ref{#1})}
\def\nn{\nonumber}

\def\({\left(\begin{array}{cccccc}}
\def\){\end{array}\right)}

\def\bes{\begin{eqnarray}}
\def\ees{\end{eqnarray}}

\newcommand{\del}{\partial}

\newcommand{\beq}{\begin{equation}}
\newcommand{\eeq}{\end{equation}}
\newcommand{\bea}{\begin{eqnarray}}
\newcommand{\eea}{\end{eqnarray}}
\newcommand{\beann}{\begin{eqnarray*}}
\newcommand{\eeann}{\end{eqnarray*}}

\newcommand{\lam}{\ensuremath{\lambda}}

\newcommand{\RR}{\mathbb{R}}

\newcommand{\s}{\ensuremath{\mathrm{s}}}

\DeclareMathOperator{\grad}{grad}
\DeclareMathOperator{\dv}{div}

\numberwithin{equation}{section}

\begin{document}

\title{New Self-similar Euler Flows: gradient catastrophe without shock formation}

\begin{abstract}
	We consider self-similar solutions to the full compressible Euler system for an ideal gas in 	
	two and three space dimensions. The system admits a 
	2-parameter family of similarity solutions depending on parameters $\lambda$ and $\kappa$. 
	Requiring locally finite amounts of mass, momentum, and energy
	imply certain constraints on $\lambda$ and $\kappa$. Further constraints are imposed for 
	particular types of flows. E.g., Guderley's pioneering construction of an unbounded converging shock 
	wave invading a quiescent fluid, requires $\kappa=0$ and $\lambda>1$.
	
	In this work we 	analyze the regime $0<\lambda<1$, which does not appear to have been
	addressed previously. Our findings include: (i) non-existence of Guderley shock solutions; 
	(ii) existence of bounded and continuous incoming similarity flows in 3-d provided $\kappa$ takes 
	the value $\hat\kappa=\frac{2(1-\lambda)}{\gamma-1}$, $\lambda$ is sufficiently small, and $\gamma$ 
	is sufficiently large; (iii) continuation of the latter flows
	beyond collapse as globally defined and continuous similarity solutions. 
	
	A key feature of these solutions is that they, in contrast to Guderley solutions,
	remain bounded at time of collapse, while the density, velocity, and sound speed
	all suffer gradient blowup. It is noteworthy that, notwithstanding infinite gradients at collapse, 
	no shock wave appears. 
	The analysis is based on a combination of analytical and numerical calculations.
\end{abstract}

\author{Helge Kristian Jenssen }\address{H.~K.~Jenssen, Department of
Mathematics, Penn State University,
University Park, State College, PA 16802, USA ({\tt
jenssen@math.psu.edu}).}

\author{Alexander Anthony Johnson}\address{A.~A.~Johnson, Department of
Mathematics, Penn State University,
University Park, State College, PA 16802, USA ({\tt
axj175@psu.edu}).}
\date{\today}
\maketitle



\tableofcontents

\section{Introduction}\label{intro}
The non-isentropic (full) compressible Euler system expresses conservation of 
mass, momentum, and energy in fluid flow in the absence of second order effects: 
\begin{align}
	\rho_t+\dv_{\bf x}(\rho \bf u)&=0 \label{mass_m_d_full_eul}\\
	(\rho{\bf  u})_t+\dv_{\bf x}[\rho {\bf  u}\otimes{\bf  u}]+\grad_{\bf  x} p&=0
	\label{mom_m_d_full_eul}\\
	(\rho E)_t+\dv_{\bf x}[(\rho E+p){\bf  u}]&=0.\label{energy_m_d_full_eul}
\end{align}
The independent variables are time $t$ and position ${\bf  x}\in\RR^n$, and the 
primary dependent variables are density $\rho$, fluid velocity ${\bf  u}$, 
and internal energy $e$; the total energy density is $E=e+\textstyle\frac{1}{2}|{\bf  u}|^2$. 
We restrict attention to ideal gases with pressure $p$ given by
\beq\label{pressure1}
	p(\rho,e)=(\gamma-1)\rho e.
\eeq
Throughout, the adiabatic constant $\gamma$ is assumed to satisfy $\gamma>1$.
The local speed of sound $c$ is
\beq\label{sound_speed}
	c=\sqrt{\textstyle\frac{\gamma p}{\rho}}=\sqrt{\gamma(\gamma-1)e}.
\eeq

We consider radial flows in space dimension $n=2$ or $n=3$, 
i.e., the flow variables depend on
position only through the distance $r=|{\bf x}|$ to the origin, and the 
velocity field is purely radial, viz.\ ${\bf u}=u\frac{\bf x}{r}$. 
With these assumptions, and within smooth regions of the flow, 
\eq{mass_m_d_full_eul}-\eq{energy_m_d_full_eul} read
\begin{align}
	\rho_t+u\rho_r+\rho(u_r+\textstyle\frac{m}{r}u) &= 0\label{m_eul}\\
	u_t+ uu_r +\textstyle\frac{1}{\gamma\rho}(\rho c^2)_r&= 0\label{mom_eul}\\
	c_t+uc_r+{\textstyle\frac{\gamma-1}{2}}c(u_r+\textstyle\frac{m}{r}u)&=0,\label{ener_eul}
\end{align}
where $\rho=\rho(t,r)$, $u=u(t,r)$, $c=c(t,r)$, and $m=n-1$.

\subsection{Self similar Euler flows}\label{ss_flows}
We next specialize further by imposing self similarity \cites{gud,cf,sed,stan}. 
For this we follow \cites{laz,cf} and introduce the similarity variables
\beq\label{alt_sim_vars}
	x=\frac{t}{r^\lambda},\qquad \rho(t,r)=r^\kappa R(x),\qquad 
	u(t,r)=-\frac{r^{1-\lambda}}{\lambda}\frac{V(x)}{x}, \qquad 
	c(t,r)=-\frac{r^{1-\lambda}}{\lambda}\frac{C(x)}{x}.
\eeq
At this stage the similarity parameters $\lambda$ and $\kappa$ are free.

Substitution of \eq{alt_sim_vars} into \eq{m_eul}-\eq{mom_eul} yield three coupled
ODEs for $R(x)$, $V(x)$, $C(x)$. The density variable $R$ can be
eliminated to give two coupled ODEs for only $V$ and $C$, viz.
\begin{align}
	\frac{dV}{dx}&=-\frac{1}{\lambda x}\frac{G(V,C)}{D(V,C)}\label{V_sim2}\\
	\frac{dC}{dx}&=-\frac{1}{\lambda x}\frac{F(V,C)}{D(V,C)},\label{C_sim2}
\end{align}
which in turn yield a single, autonomous ODE 
\beq\label{CV_ode}
	\frac{dC}{dV}=\frac{F(V,C)}{G(V,C)}
\eeq
relating $V$ and $C$ along self-similar solutions. The functions $D$, $G$, $F$ are given by
\begin{align}
	D(V,C)&=(1+V)^2-C^2\label{D}\\
	G(V,C)&=nC^2(V-V_*)-V(1+V)(\lam+V)\label{G}\\
	F(V,C)&=C\left\{C^2\big(1+\textstyle\frac{\alpha}{1+V}\big)
	-k_1(1+V)^2+k_2(1+V)-k_3\right\},\label{F}
\end{align}
where
\beq\label{V_*}
	V_*=\textstyle\frac{\kappa-2(\lambda-1)}{n\gamma},
\eeq
\beq\label{alpha}
	\alpha=\textstyle\frac{1}{2\gamma}[\kappa(\gamma-1)+2(\lambda-1)],
\eeq
and
\beq\label{ks}
	k_1=1+{\textstyle\frac{(n-1)(\gamma-1)}{2}},\qquad 
	k_2={\textstyle\frac{(n-1)(\gamma-1)+(\gamma-3)(\lam-1)}{2}},\qquad
	k_3=\textstyle\frac{(\gamma-1)(\lam-1)}{2}.
\eeq

Evidently, the construction of radial self-similar Euler flows requires an analysis
of the phase portrait for \eq{CV_ode} in the $(V,C)$-plane. However, since the ODEs
\eq{V_sim2}-\eq{C_sim2} are singular along the two critical lines defined by $D(V,C)=0$, 
only certain trajectories of \eq{CV_ode} yield physically meaningful flows. Specifically, 
any trajectory crossing a critical line can do so only at points where all three 
of $F$, $G$, and $D$ vanish. 

Having identified an admissible trajectory $\Gamma$ of \eq{CV_ode}
connecting some of its equilibria, it may be used to generate a solution to the 
system \eq{V_sim2}-\eq{C_sim2}. Finally, it must be checked that the resulting 
solution $(V(x),C(x))$ moves along $\Gamma$ in the correct manner as $x$ increases 
from $-\infty$ to $+\infty$. In particular, this is an issue for the solutions we construct.
These pass through the origin in the $(V,C)$-plane for $x=0$ and it must be verified 
that the signs of $F$, $G$, $D$, and $x$ match up correctly at the crossing. E.g., 
the physical requirement that sound speed is non-negative implies that the 
solution passes from $\{C>0\}$ to $\{C<0\}$ as $x$ increases from negative to positive values.

The analysis of \eq{CV_ode} involves
a fair amount of calculations as the number, locations, and types of its equilibria 
depend on the parameters $n$, $\gamma$, $\lambda$, and $\kappa$. 
In Section \ref{crit_pts} we record the explicit expressions for the equilibria, valid for any choice
of parameters. Since we have not found it in the existing literature, we also provide a complete 
breakdown of when the various equilibria of \eq{CV_ode} are present.

Once a solution to \eq{V_sim2}-\eq{C_sim2} has been selected, the velocity and sound 
speed in the corresponding Euler flow are determined via \eq{alt_sim_vars}. 
The full description of the flow requires also the density 
field $\rho(t,r)=r^\kappa R(x)$. This can be obtained within each region of continuous 
flow from the following explicit {\em entropy integral}
\beq\label{entr_int}
		\textstyle\left(\frac{C(x)}{x}\right)^2\!R(x)^{1-\gamma}[R(x)|1+V(x)|]^q\equiv \text{constant $>0$},
\eeq  
with $R\geq 0$ and  
\beq\label{q}
		q=\textstyle\frac{1}{\kappa+n}[\kappa(\gamma-1)+2(\lambda-1)]
		= \frac{2\gamma}{\kappa+n}\alpha.
\eeq 
The existence of this integral is a consequence of the fact that the specific entropy remains 
constant along particle trajectories in continuous Euler flow.

For later reference we note that the choice
\beq\label{isentr_kappa}
	\kappa=\hat\kappa:=\textstyle\frac{2(1-\lambda)}{\gamma-1},
\eeq
makes $\alpha$, and hence $q$ vanish.
The entropy integral \eq{entr_int} then reduces to
\beq\label{CR}
	(\textstyle\frac{C(x)}{x})^2R(x)^{1-\gamma}\equiv \text{constant $>0$.}
\eeq
In terms of temperature $\theta\propto c^2$ and density $\rho$, this amounts to
$\theta\rho^{1-\gamma}$ being constant, i.e., the specific entropy 
takes a constant value throughout any region of continuity. 
Thus, continuous similarity flows with $\kappa=\hat\kappa$ provide
isentropic solutions to the Euler system.
The solutions we build in Section \ref{cont_constrcn} are of this type.

\subsection{Outline and main results}\label{main_findings}
The present work addresses a particular type of converging-diverging flows 
with $\lambda\in(0,1)$,
in which an incoming radially symmetric wave collapses on the center of motion
and reflects an outgoing wave. Without loss of generality, the time of collapse 
is chosen as $t=0$, a choice which is built into the definition of the similarity 
variable $x$ in \eq{alt_sim_vars}.

Before describing our findings we briefly review some earlier results.
The pioneering study \cite{gud} of Guderley provided examples
of unbounded converging-diverging shock waves in an ideal gas. 
An incoming spherical shock wave approaches the origin by invading a 
quiescent fluid (homogeneous and at rest),
while gaining strength. At collapse it has infinite speed and the 
velocity, sound speed, and pressure in its immediate wake are unbounded. The subsequent flow 
accommodates the infinite amplitudes at the center of motion by generating 
an expanding shock wave, which then slows down and weakens as it interacts with the 
still-incoming flow ahead of it. 

In what follows, solutions in which a converging shock invades a quiescent fluid, 
collapses at the origin, and then generates an expanding shock wave, will be 
referred to as {\em Guderley solutions}. Their construction depends on resolving a
nonlinear eigenvalue problem for the similarity parameter $\lambda$ (see 
Section \ref{no_Guderley}). It turns out that, in a Guderley solution,
the similarity parameter $\kappa$ must necessarily be zero and that the temperature 
in the quiescent part of the fluid vanishes identically. 
The allowed $\lambda$ values depend on both the geometry ($n$) and the gas
($\gamma$), and are dictated by the requirement that a certain ODE-trajectory pass
through a particular equilibrium of \eq{CV_ode}. Their determination must be done numerically,
a task that has been carried out to considerable accuracy in a number of works 
(for $n=2$ or $3$ and various $\gamma>1$); see \cites{rkb_12,gud,bk,laz,am,haf,hg}
and references therein.

\begin{remark}
	In the applied literature on self-similar Euler flows the emphasis 
	has been on Guderley solutions due to their relevance to inertial fusion 
	research, \cites{am,dm,p,mtv_s}.  The closely related construction of unbounded 
	self-similar cavity flows has been analyzed in \cites{bk,laz,hun_63}. 
\end{remark}

In all works on self-similar solutions to the Euler system that we are aware of, attention is restricted to 
similarity parameters $\lambda>1$ or, as a limiting case, $\lambda=1$\footnote{We 
note that $\lambda=1$ provides the setting for the study of multi-d Riemann 
problems, \cite{zheng}.}. Our first objective in this work is to consider the possibility  
of Guderley shock solutions when $\lambda\in(0,1)$. Since a shock in a similarity flow 
propagates along a path with 
$x=\frac{t}{r^\lambda}\equiv constant$, $\lambda\in(0,1)$ would yield a ``glancing'' 
shock wave that weakens and slows down, reaching the center of motion with vanishing 
speed. However, as described in Section \ref{no_Guderley}, it does not appear 
possible to generate a Guderley solution when $\lambda\in(0,1)$: the relevant ODE-trajectories
simply do not reach the required equilibrium.  

We then turn to the possibility of constructing {\em continuous} self-similar
radial Euler flows. For $\lambda>1$ such solutions have recently been  
constructed, up to time of collapse, in the works \cites{jt1,jt2,mrrs1}. 
These solutions suffer amplitude blowup at the $t=0$ and are propagated to positive 
times in \cites{jt1,jt2} by having a shock emerge from the center of motion, similar to 
what occurs in Guderley solutions. 
\begin{remark}
	The work \cite{mrrs1} addresses the subtle issue of constructing {\em smooth} ($C^\infty$) self-similar
	isentropic flows (up to collapse). The recent work \cite{biasi} provides numerical 
	evidence that these solutions are linearly unstable with respect to 1-d radial perturbations.
	
	We note that the continuous 
	solutions considered in \cites{jt1,jt2,mrrs1} demonstrate in particular that amplitude blowup does not 
	require a central region of vanishing pressure, as is the case in Guderley solutions and 
	cavity flows.
\end{remark}

The main contribution of the present work is the construction and analysis of {\em globally} 
continuous radial self-similar flows for the full Euler system with similarity parameter $\lambda\in(0,1)$. 
This parameter range yields very different behavior compared to those 
of Guderley solutions, or those in \cites{jt1,jt2,mrrs1}: instead of suffering amplitude 
blowup, the primary flow variables $\rho$, $u$, $c$ remain bounded near the center of motion,
and instead suffer {\em gradient catastrophes} at time of collapse $t=0$. 
However, notwithstanding the infinite gradients, 
the solutions propagate as {\em continuous} flows to positive times. We find it noteworthy 
that this can occur even in cases where all fluid particles move
{\em toward} the origin at time $t=0$. 

The issue of shock formation and propagation in multi-d Euler flows has recently been analyzed in great detail,
providing fundamental new results in the field, see \cites{christ_1,christ_2,ls_1,ls_2,bi,bsv_1,bdsv} 
and references therein. In this connection, the solutions we obtain here simply point out that
singularity formation (i.e., some of the primary flow variables suffer a gradient catastrophe), 
does not necessarily give rise to a shock wave; for further detail see Remark \ref{no_shock}.   

\begin{remark}
	We have not addressed the stability of the solutions we obtain. However, we note
	that their pressure fields do not suffer gradient blowup.
	In fact, at time of collapse the pressure 
	vanishes super-linearly as $r\downarrow0$ (see Section \ref{r=0_behavior}), 
	which might provide a stabilizing effect.
\end{remark}

The construction of globally continuous self-similar flows with $\lambda\in(0,1)$ follows the 
standard strategy of building solutions from trajectories of the  
ODE \eq{CV_ode} connecting some of its equilibria. However, the requirements of  
continuity and $0<\lambda<1$ impose additional constraints. First, as in \cite{jt2}, 
we show that continuity of the flow
(specifically, boundedness of $\rho$ and $c$ at the center of motion prior to collapse)
requires the similarity parameter $\kappa$ to take the ``isentropic'' value $\kappa=\hat \kappa$ in \eq{isentr_kappa}.
As noted above, this choice renders the flow globally isentropic. We verify that it also guarantees the
absence of a gradient catastrophe prior to $t=0$ (Section \ref{r=0_behavior}).

In addition, to guarantee the existence of suitable trajectories when $\lambda\in(0,1)$,
further restrictions must be imposed. 
These are dictated by the requirement that a certain critical point ($P_8$ in what follows) be 
a proper node with a suitable primary direction.
It turns out that this requires the space dimension to be $3$, and that $\lambda$ belongs to the 
restricted range $(0,\frac{1}{9})$.
Finally, the adiabatic constant needs to be sufficiently large, viz.\ $\gamma>\gamma_3(\lambda)$, where the latter 
is an increasing function satisfying 
\[\lim_{\lambda\downarrow 0}\gamma_3(\lambda)=\gamma_*\approx 8.72,\qquad 
\lim_{\lambda\uparrow \frac 1 9}\gamma_3(\lambda)=+\infty;\]
see Figure \ref{gamma_3}.
With these assumptions met, we verify numerically the existence of suitable ODE trajectories.
%
Our main findings are as follows:

\medskip

\noindent{\bf Main Results.} {\em Consider radial self-similar solutions of the form \eq{alt_sim_vars} to 
the full multi-d Euler system \eq{mass_m_d_full_eul}-\eq{energy_m_d_full_eul} in space 
dimension $2$ or $3$, with similarity variables $\lambda$ and $\kappa$. Then:
\begin{enumerate}
	\item No Guderley solutions (converging shock invading a quiescent state) appear possible
	when $\lambda\in(0,1)$.
	\item The existence of continuous self-similar solutions requires that $\kappa$ 
	takes the ``isentropic'' value $\hat\kappa$ in \eq{isentr_kappa}; in turn, this 
	choice renders the flow globally isentropic and without singularities (gradient 
	catastrophes) prior to collapse.
	\item With $\kappa=\hat\kappa$ there is a 1-parameter family of continuous 
	self-similar solutions \eq{alt_sim_vars} which describe a  converging 
	wave collapsing at the origin at time $t=0$. Our construction of this type of solution requires
	$n=3$, $\lambda\in(0,\frac 1 9)$, and sufficiently large values of $\gamma$, viz.\ $\gamma>\gamma_3(\lambda)$. 
	\item The solutions described in {\em (3)}, while locally bounded, are such that $\rho$, $u$, $c$ 
	all suffer gradient catastrophes at the origin at time of collapse.
	The pressure is $C^1$-smooth and vanishes
	super-linearly as the center of motion is approached at time $t=0$.
	\item Notwithstanding infinite gradients in $\rho$, $u$, and $c$ at collapse, we provide examples 
	of solutions that extend as continuous similarity solutions to positive times. 
	Numerical evidence suggests that no outgoing shock is generated whenever
	$\lambda$ and $\gamma$ are as described in {\em (3)}. 
\end{enumerate}}

\medskip

Two remarks are in order.
\begin{remark}
	The solutions described in parts {\em (3)-(5)} have locally bounded
	mass, momentum, and energy. On the other hand, it is readily verified that they 
	are unbounded as $r\to\infty$ at any fixed time $\bar t$; specifically,
            \beq\label{uc_behav}
            	u(\bar t,r)\sim r^{1-\lambda},\, c(\bar t,r)\sim r^{1-\lambda}
		\quad\text{and}\quad \rho(\bar t,r)\sim r^{\hat\kappa}
		\qquad\text{as $r\uparrow\infty$.}
            \eeq
        Since $\lambda<1$, the solutions have infinite total mass, momentum, 
        and energy. However, it appears reasonable that the same local behavior near
        the center of motion can be obtained in solutions with bounded mass, 
        momentum, and energy. This could be achieved by fixing a time $t_0<0$ 
        and modifying the self-similar solution outside of a sufficiently large ball $B_{R_0}(0)$. 
        Specifically, $R_0$ should be larger than the radial position $r_c(t_0)$, where $r_c(t)$ 
        denotes the critical 1-characteristic (sonic curve) passing through the origin 
        at $t=0$. This would ensure that the modified part of the solution remain causally 
        independent of the flow near $r=0$, provided the modified solution remains continuous 
        up to time $t=0$. 
        It is reasonable that this scenario can be achieved (e.g., by having 
        the modification at time $t_0$ generate a suitable expanding 
        rarefaction wave), but we stress that we do not 
        have a rigorous proof of this.
\end{remark}
\begin{remark}\label{no_shock}
	Concerning the absence of shocks, it is of interest to consider the behavior 
	of 1-characteristics near the center of 
	motion in the  continuous solutions described above. For this, fix a time $t_0<0$
	and let $r(t,\xi)$ denote the 1-characteristic that passes through location $r=\xi$
	at time $t_0$, i.e., 
	\beq\label{1_char}
		\del_t r(t,\xi)=(u-c)|_{(t,r(t,\xi))},\qquad r(t_0,\xi)=\xi.
	\eeq
	The critical 1-characteristic
	(sonic line) which arrives at the origin at time of collapse, propagates along 
	the path $r_c(t)=|\frac{t}{x_8}|^\frac{1}{\lambda}$ for $t<0$, where $x_8$ 
	is the $x$-value for which the incoming  solution passes through the particular critical 
	point $P_8\in\{D=0\}\cap\{F=0\}\cap\{G=0\}$ (cf. \eq{D}-\eq{G}\eq{F}). We are interested in the density of 1-characteristics 
	at the center of motion at time $t=0$. We therefore set
	\[\mu(t,\xi):=\del_\xi r(t,\xi),\]
	and seek to compute $\mu(0,\xi_c)$, where $\xi_c=r_c(t_0)$.
	Shock formation is expected when the characteristics concentrate,
	i.e., $\mu(0,\xi_c)=0$.
	Differentiating \eq{1_char} with respect to $\xi$ yields 
	\[\del_t\log \mu(t,\xi)=(u_r-c_r)|_{(t,r(t,\xi))}
	=\textstyle\frac{1}{\lambda t}[\lambda x(V'(x)-C'(x))+(C(x)-V(x))].\]
	For $\xi=\xi_c$, which corresponds to $x=x_8$, this gives
	\[\del_t\log \mu(t,\xi_c)=\textstyle\frac{A_8}{\lambda t},\]
	where the constant $A_8$ is given by 
	\[A_8=\lambda x_8(V'(x_8)-C'(x_8))+1.\]
	$A_8$ is explicitly available and is given in terms of the first partials of $F(V,C)$ and $G(V,C)$ at $P_8$, and 
	is a (somewhat complicated) function of $\lambda$ and $\gamma$.
	Integrating from time $t_0$ to $t<0$, and using $\mu(t_0,\xi)\equiv 1$, we have
	\[\mu(t,\xi_c)=\big|\textstyle\frac{t}{t_0}\big|^\frac{A_8}{\lambda}.\]
	It follows that shock formation at the center of motion at time $t=0$ would require $A_8>0$.
	However, a numerical evaluation reveals that $A_8<0$ whenever 
	the parameters are as described in part (3) of the Main Results (i.e., $n=3$, $\lambda\in(0,\frac 1 9)$,
	and $\gamma>\gamma_3(\lambda)$). This provides an 
	analytic verification of the absence of shocks in the constructed self-similar flows.
	An alternative, graphic verification based on the Rankine-Hugoniot relations is described in 
	Section \ref{no_shock_graph}.
\end{remark}
The rest of the article is organized as follows. In Section \ref{crit_pts} we record the 
equilibria of \eq{V_sim2}-\eq{C_sim2}. There are up to 11 of these and we provide a complete breakdown of their  
presence depending on the parameters $n$, $\gamma$, $\lambda$, and $\kappa$. 
The cases when $\kappa=0$ and $\kappa=\hat\kappa$ are treated separately for later use. 
Section \ref{no_Guderley} describes
Guderley solutions and argues that no such solution appears possible
when $\lambda\in(0,1)$. Turning to the construction of continuous similarity flows 
for this $\lambda$ range, we make use of the singular points at infinity ($P_{\pm\infty}$) and at the origin ($P_1$). 
For the resulting flows we then analyze the restrictions placed on $\lambda$ and $\kappa$ 
by integrability and continuity constraints. These are dealt with in Section \ref{lam_kap_constrs}
where it is found that the latter constraint fixes $\kappa=\hat\kappa$. 
We also verify that no gradient catastrophe occurs prior to collapse in the resulting flows.

The construction of the relevant trajectories 
is detailed in Section \ref{cont_constrcn}. For this we want that one of the critical 
points, $P_8$, is a proper node, guaranteeing that an infinite number of trajectories are drawn to it. This requires 
a detailed analysis of various quantities defined in terms of the partial derivatives of $F$ and $G$ at $P_8$. We 
then show how the requirement that the saddle point $P_{+\infty}$ be joined to the node at $P_8$ via
a trajectory $\Gamma_1$ of  \eq{V_sim2}-\eq{C_sim2} imposes the additional constraints $n=3$, $\lambda\in(0,\frac 1 9)$,
and $\gamma>\gamma_3(\lambda)$ (Sections \ref{Gamma_1}-\ref{P_8_analysis}). 
We next describe how to select suitable trajectories $\Gamma_2$ joining $P_8$ to its reflection $P_9$ about the $V$-axis
(Section \ref{P_1_analysis}). 
Such trajectories must pass through the proper node $P_1$ at the origin; there is typically an infinite number of such solutions.
Finally, we add the reflection $\Gamma_3$ of $\Gamma_1$ about the $V$-axis to define the 
complete solution trajectory $\Gamma:=\Gamma_1\cup\Gamma_2\cup\Gamma_3$ of  \eq{V_sim2}-\eq{C_sim2}.
The corresponding flow variables defined via \eq{alt_sim_vars} and \eq{CR} then provide global, $3$-dimensional, 
self-similar, and continuous Euler flows. Section \ref{sum} summarizes the  construction 
and the required numeric tests, which are done in Section \ref{Gamma_num_verfcn}.

To illustrate the construction in Section \ref{cont_constrcn}, we provide figures
displaying the various trajectories for the particular case $n=3$, $\lambda=0.02$, $\gamma=12$. 
We finally illustrate graphically the absence of a shock wave in the flow 
after collapse in this case (Section \ref{no_shock_graph}). 

\section{Critical points}\label{crit_pts}
Throughout this section $n=2$ or $3$, and $\gamma>1$. 
The goal is to identify the critical points of \eq{CV_ode} and to determine how 
their presence depends on the parameters $\kappa$ 
and $\lambda$, which are unrestricted for now (until Section \ref{isntr_sing_pnts}).
As we have not found it in the existing literature, we provide a complete 
breakdown of all the cases.
\begin{remark}
	The presence of some of the critical points (viz.\ $P_6$-$P_9$ in the 
	notation introduced below) places certain constraints
	on the parameters $\lambda$, $\kappa$, $\gamma$, and $n$; 
	see Section \ref{P4_P9}. Further requirements 
	are imposed in Section \ref{lam_kap_constrs}.
\end{remark}
We introduce the {\em critical lines}
\[\mathcal L_\pm:=\{(V,C)\,|\,C=\pm(1+V)\},\]
and note the relation
\beq\label{non_obvious_reln}
	F(V,\pm(1+V))\equiv \mp\textstyle\frac{(\gamma-1)}{2}G(V,\pm(1+V)).
\eeq
The critical points of \eq{CV_ode} are the points of intersection between
the zero-level sets 
\[\mathcal F:=\{(V,C)\,:\, F(V,C)=0\}\qquad\text{and}\qquad 
\mathcal G:=\{(V,C)\,:\, G(V,C)=0\}\]
of the functions $F$ and $G$ defined in \eq{F} and \eq{G}, respectively. 
Note that $V=V_*$ (see \eq{V_*}) is a vertical asymptote for $\mathcal G$.
The following symmetries will be important in assembling trajectories 
of \eq{V_sim2}-\eq{C_sim2},
\beq\label{symms}
	G(V,-C)=G(V,C),\qquad F(V,-C)=-F(V,C).
\eeq

It turns out that there are up to nine points of intersection between 
$\mathcal F$ and $\mathcal G$, 
and we follow \cite{laz} in numbering these $P_i=(V_i,C_i)$, $i=1,\dots,9$. 
In addition there are two critical points at infinity,
\beq\label{p+-}
	P_{\pm\infty}:=(V_*,\pm\infty),
\eeq
both of which are used in the construction of continuous Euler 
flows in Section \ref{cont_constrcn}.

\subsection{Critical points $P_1$-$P_3$}\label{P1_P3}
We begin by observing that there are always three critical points located 
along the $V$-axis:
\[P_1:=(0,0), \qquad P_2:=(-1,0), \qquad\text{and}\qquad P_3:=(-\lambda,0).\]  
Of these only $P_1$ is relevant for our purposes. The linearization of 
\eq{CV_ode} at $P_1$ is $\frac{dC}{dV}=\frac C V$ (for all values of $n$, 
$\gamma$, $\kappa$, and $\lambda$), showing that $P_1$ is a 
star point (proper node). Thus, for any straight line $\ell$ from the origin, there is a unique 
trajectory $(V,C(V))$ of \eq{CV_ode} which approaches the origin tangent to $\ell$.

Assume now that a solution $(V(x),C(x))$ of \eq{V_sim2}-\eq{C_sim2} approaches 
$P_1$ with slope $k$. The corresponding trajectory $(V,C(V))$ of \eq{CV_ode}
then satisfies $C(V)\approx kV$ for $V\approx 0$, and an inspection of 
\eq{V_sim2}-\eq{C_sim2} yields
\[\frac{dV}{dx}\approx \frac V x \qquad\text{and}\qquad \frac{dC}{dx}\approx \frac C x\]
as $P_1$ is approached. It follows from this that any solution of \eq{V_sim2}-\eq{C_sim2} 
reaching $P_1$ must do so for $x=0$, and also that the limits
\beq\label{well_bhvd}
	\nu:=\lim_{x\to0}{\textstyle\frac{V(x)}{x}}\qquad\text{and}\qquad
	\omega:=\lim_{x\to0}\textstyle\frac{C(x)}{x}\qquad\text{exist as finite numbers.}
\eeq
This last property is a minimal requirement 
for \eq{alt_sim_vars} to yield a meaningful flow at time $t=0$.
(In the limiting case that $P_1$ is reached with infinite slope, $\nu$ vanishes.)

Observe also that due to the requirement that $c(t,r)\geq 0$, we get from 
\eq{alt_sim_vars} that a solution $(V(x),C(x))$ of \eq{V_sim2}-\eq{C_sim2}
must necessarily pass from the upper half-plane $\{C>0\}$ to the lower half-plane
$\{C<0\}$ as $x$ increases from negative to positive values.

\subsection{Critical points $P_4$-$P_9$}\label{P4_P9}
The critical points $P_4$-$P_9$ are obtained 
by solving $G(V,C)=0$ for $C^2$ in terms of $V$, and substituting the result into
the equation $F(V,C)=0$; this yields a cubic polynomial in $V$ (see below). According to the 
symmetries in \eq{symms}, the critical points $P_4$-$P_9$ come in pairs
located symmetrically about the $V$-axis. The ones located above (below) the $V$-axis 
are $P_4$ ($P_5$), $P_6$ ($P_7$), and $P_8$ ($P_9$). It turns out that among
these, $P_4$ and $P_5$ are present for all values of $n$, $\kappa$, $\lambda$, and
$\gamma$, while $P_6$-$P_9$ may or may not be present.

Restricting attention to $P_4$, $P_6$, and $P_8$, we proceed to determine when 
and where these occur. From $G(V,C)=0$ we have 
\beq\label{C^2_G}
	C^2=\textstyle\frac{V(1+V)(\lambda+V)}{n(V-V_*)}.
\eeq
Substituting \eq{C^2_G} into $F(V,C)=0$, and recalling that we now seek 
critical points off the $V$-axis, give the following 
cubic equation for $W:= 1+V$:
\begin{align*}
	&[nk_1-1] W^3 - [nk_2-\beta k_1+\alpha+(\lambda-2)] W^2\\
	&\qquad\qquad+[nk_3-\beta k_2-(\lambda-2)\alpha+(\lambda-1)]W+[\beta k_3+(\lambda-1)\alpha]=0,
\end{align*}
where $\alpha$ and the $k_i$ are given in \eq{alpha} and \eq{ks}, and $\beta=-n(1+V_*)$.
This cubic always has one real root, denoted $W_4$,
and two possibly complex roots $W_6$ and $W_8$. The root $W_4=1+V_4$ is given by
\beq\label{V_4}
	V_4=-\textstyle\frac{\lambda}{1+\frac{n}{2}(\gamma-1)},
\eeq
(cf.\ Eqn.\ (3.3) in \cite{laz}).
We note that $V_4$ is independent of $\kappa$; however, the corresponding $C$-value $C_4>0$,
given by \eq{C^2_G}, does depend on $\kappa$ through $V_*$.
The two remaining roots $V_6\equiv V_-$ and $V_8\equiv V_+$ are given by
\begin{align}
	&V_\pm=\textstyle\frac {1}{2 m \gamma} 
	\Big[(\gamma-2)\mu +\kappa-m\gamma\pm\sqrt{(\gamma-2)^2\mu^2
	-2[\gamma m(\gamma+2)-\kappa(\gamma-2)]\mu+(\gamma m+\kappa)^2}\Big],\label{V_pm}
\end{align}
where we have set
\beq\label{mmu}
	m:=n-1\qquad\text{and}\qquad\mu:=\lambda-1.
\eeq
(In what follows we use either $\lambda$ or $\mu$, always assuming $\mu=\lambda-1$.)
The values $V_\pm$,  when real, yield the critical points $P_6=(V_6,C_6)$ 
and $P_8=(V_8,C_8)$ above the $V$-axis via \eq{C^2_G}. 
We record the non-obvious fact that
\beq\label{on_L_+}
	C_i^2=(1+V_i)^2\qquad\text{for $i=6,8$.}
\eeq
Therefore, whenever $P_6$ and $P_8$ are present, they are necessarily located 
on one of the critical lines $\mathcal L_\pm$. By symmetry, the same applies to $P_7$ and $P_9$.

We proceed to determine when the critical points $P_6$ and $P_8$ are present. This 
amounts to deciding when $V_\pm$ are real, i.e., when the
radicand in \eq{V_pm} is non-negative. To do so we consider two situations: either $\kappa$ 
is a free parameter, or $\kappa=\hat\kappa$ is given in terms of $\lambda$ and $\gamma$
by \eq{isentr_kappa}. We start with the general case where $\kappa$ is 
independent of $\lambda$, $\gamma$, and $n$.

\subsubsection{General case: $\kappa$ free.}
Consider the radicand in \eq{V_pm} as a polynomial
in $\mu=\lambda-1$; to organize the analysis 
we consider four sub-cases: 
\begin{enumerate}
	\item[(i)] For $\gamma=2$ the radicand in \eq{V_pm} is linear in $\mu$, with the single root
	corresponding to
	\beq\label{lam_pm_gamma=2}
		\lambda=\lambda_{\max}:=1+\textstyle\frac{(2m +\kappa)^2}{16 m}.
	\eeq
	In this case, $V_\pm$ are real if and only if $\lambda\leq \lambda_{\max}$. 
	The limiting case $\lambda=\lambda_{\max}$ yields $V_+=V_-=\frac{\kappa}{4m}-\frac 1 2$.
\end{enumerate}
Next, a direct calculation shows that for $\gamma\neq2$ the radicand in \eq{V_pm} has the roots
\beq\label{lam_pm}
	\lambda=1+\textstyle\frac{(m\gamma +\kappa)^2}{\left(\gamma\sqrt{m}\pm 
	\sqrt{2\gamma m-\kappa(\gamma-2)}\right)^2}.
\eeq
Here the $\pm$ signs are unrelated to those in \eq{V_pm}. The expressions in \eq{lam_pm} 
generalize the expressions recorded by Lazarus who treated the cases $\kappa=0$ and
$\kappa=\hat\kappa$ (Section 3 in \cite{laz}).
\begin{enumerate}
	\item[(ii)] When $\gamma\neq 2$ and the radicand $2\gamma m-\kappa(\gamma-2)$ 
	in \eq{lam_pm} is strictly negative, then the radicand in \eq{V_pm} has no real $\mu$-root. 
	Therefore, since the coefficient of $\mu^2$ in \eq{V_pm} is positive, the radicand in \eq{V_pm} is 
	then strictly positive. Consequently, $V_\pm$ are necessarily real and distinct numbers 
	in this case. 
\end{enumerate}
If $\gamma\neq 2$ and the radicand in \eq{lam_pm} satisfies $2\gamma m-\kappa(\gamma-2)\geq 0$, 
there are two further sub-cases depending on whether the minus-sign in \eq{lam_pm} gives a 
vanishing denominator:
\begin{enumerate}
	\item[(iii)] When $\gamma\neq 2$ and $\kappa=-\gamma m$ (in particular, the radicand 
	in \eq{lam_pm} is strictly positive, but the minus-sign gives a $\frac{0}{0}$ expression), 
	substitution of the latter $\kappa$-value directly into \eq{V_pm} gives 
	\beq\label{V_pm_iii}
		V_\pm=\textstyle\frac {1}{2 m \gamma}\Big[(\gamma-2)\mu-2m\gamma
		\pm\sqrt{(\gamma-2)^2\mu^2-4\gamma^2 m\mu}\Big].
	\eeq
	In this case $V_\pm$ are real numbers if and only if $\mu\leq 0$ or 
	$\mu\geq\frac{4m\gamma^2}{(\gamma-2)^2}$, i.e., if and only if 
	\beq\label{min_max_lambda_1}
		\lambda\leq\lambda_{\max}:=1\qquad\text{or}\qquad
		\lambda\geq\lambda_{\min}:=1+\textstyle\frac{4m\gamma^2}{(\gamma-2)^2}.
	\eeq
	We have that $V_+=V_-$ if and only if $\lambda$ takes one of the values 
	$\lambda_{\min}$ or $\lambda_{\max}$.
\end{enumerate}
\begin{enumerate}
	\item[(iv)] Finally, consider the case when $\gamma\neq 2$, $\kappa\neq-\gamma m$, 
	and the radicand $2\gamma m-\kappa(\gamma-2)$ in \eq{lam_pm} is non-negative. We set
        \beq\label{lam_max}
        		\lambda_{\max}:=1+\textstyle\frac{(m\gamma +\kappa)^2}{\left(\gamma\sqrt{m}
        		+\sqrt{2\gamma m-\kappa(\gamma-2)}\right)^2},
        \eeq
        and
        \beq\label{lam_min}
    		\lambda_{\min}:=1+\textstyle\frac{(m\gamma +\kappa)^2}{\left(\gamma\sqrt{m}
        		-\sqrt{2\gamma m-\kappa(\gamma-2)}\right)^2},
        \eeq
        so that $V_\pm$ are real if and only if, either 
	\[\lambda\leq\lambda_{\max}\qquad\text{or}\qquad\lambda\geq\lambda_{\min}.\]
	Again, $V_+=V_-$ if and only if $\lambda$ takes one of the values 
	$\lambda_{\min}$ or $\lambda_{\max}$.
\end{enumerate}

\subsubsection{The case $\kappa=0$}\label{lam_max_min_kappa=0}
For later use we consider separately the case when $\kappa=0$. 
$V_\pm$ are then real provided
\[\lambda\leq\lambda_{\max}=1+\textstyle\frac{m\gamma}{\left(\sqrt{\gamma}+ \sqrt{2}\right)^2}
\qquad\text{or}\qquad
\lambda\geq\lambda_{\min}=1+\textstyle\frac{m\gamma}{\left(\sqrt{\gamma}- \sqrt{2}\right)^2}.\]
For the special value $\gamma=2$, we have $V_\pm$ real whenever
\[\lambda\leq\lambda_{\max}=1+\textstyle\frac{m}{4}.\]
Note that when $\kappa=0$, we necessarily have $\lambda_{\max}>1$.

\begin{figure}
	\centering
	\includegraphics[width=8cm,height=8cm]{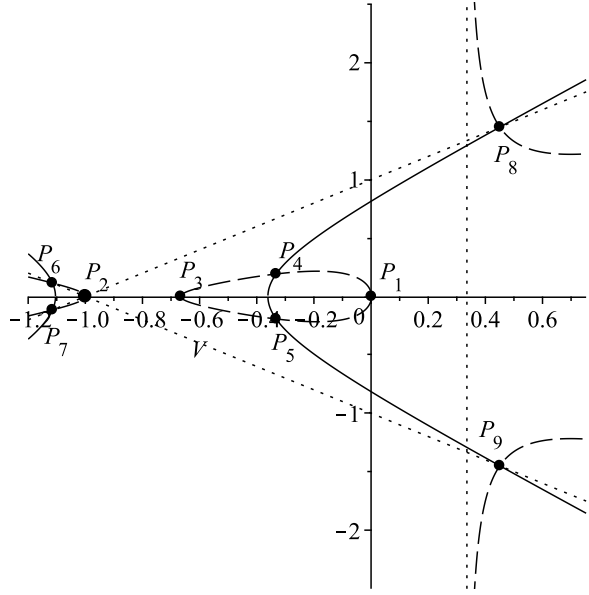}
	\caption{The zero-level curves  
	of $F(V,C)$ (solid, including the $V$-axis) and $G(V,C)$ (dashed),
	together with the critical lines $\mathcal L_\pm=\{C=\pm(1+V)\}$ and the vertical asymptote 
	$V=V_*$ (dotted). The parameters are $n=3$, $\gamma=\frac{5}{3}$,
	$\lambda=\frac{2}{3}$, and $\kappa=\hat\kappa=1$. All of the singular points
	$P_1$-$P_9$ are present in this case (solid dots).}\label{crit_points}
\end{figure}

\subsubsection{Isentropic case: $\kappa=\hat\kappa$.}\label{isntr_sing_pnts}
In this case $\kappa$ is fixed according to \eq{isentr_kappa} for 
given $\lambda$ and $\gamma$. In terms of $\mu=\lambda-1$ we have 
\[\hat\kappa=-\textstyle\frac{2\mu}{\gamma-1},\]
and substitution of this $\kappa$-value into \eq{V_pm} gives
\beq\label{isntr_V_pm}
	V_\pm=\textstyle\frac{1}{2}(a\pm\sqrt{Q}),
\eeq
where 
\beq\label{aQ}
	a=\textstyle\frac{(\gamma-3)}{m(\gamma-1)}\mu-1\qquad\text{and}\qquad
	Q=\left(\textstyle\frac{(\gamma-3)}{m(\gamma-1)}\right)^2\!\!\mu^2
	-2\textstyle\frac{(\gamma+1)}{m(\gamma-1)}\mu+1.
\eeq
To have $V_\pm$ real requires $Q\geq 0$. Regarding $Q$ as a polynomial in $\mu$
there are two cases:
\begin{enumerate}
	\item[(a)] When $\gamma=3$, $Q$ is linear in $\mu$ and $Q\geq 0$ if and only if
	$\mu\leq \frac{m}{4}$. In terms of $\lambda$ this means that $V_\pm$ are real 
	if and only if $\lambda\leq \lambda_{\max}:=1+ \frac{m}{4}$. Also, $V_-=V_+$
	if and only if $\lambda=\lambda_{\max}$.
	\item[(b)] For $\gamma\neq3$, $Q$ is a quadratic in $\mu$ with a positive leading
	coefficient. A direct calculation shows that $Q\geq0$ if and only if, either 
	$\lambda\leq \lambda_{\max}$ or $\lambda\geq\lambda_{\min}$, where 
	\beq\label{lambda_max}
		\lambda_{\max}=1+\textstyle\frac{m(\gamma-1)}{(\gamma+1)+\sqrt{8(\gamma-1)}}
	\eeq
	and 
	\beq\label{lambda_min}
		\lambda_{\min}=1+\textstyle\frac{m(\gamma-1)}{(\gamma+1)-\sqrt{8(\gamma-1)}}.
	\eeq
	Finally, $V_+=V_-$ if and only if $\lambda$ takes one of the values 
	$\lambda_{\min}$ or $\lambda_{\max}$.
\end{enumerate}
We note that, in either case (a) or case (b), $\lambda_{\max}>1$ holds due to our assumptions 
$n\geq 2$ and $\gamma>1$. In particular, when $\kappa=\hat\kappa$, $P_6$ and $P_8$ 
are present whenever $\lambda<1$.

Figure \ref{crit_points} displays a representative case with $\kappa=\hat\kappa$
and all critical points present.

\subsection{Critical points $P_{\pm\infty}$}\label{P_infs}
The critical points at infinity are $P_{\pm\infty}=(V_*,\pm\infty)$.
To analyze these we consider $P_{+\infty}$ (sufficient according to \eq{symms})
and change to the variables $W:=V-V_*$ and $Z:=C^{-2}$. Linearizing the 
resulting equation for $\frac{dZ}{dW}$ about $(W,Z)=(0,0)$ yields
\beq\label{WZ_ode}
	\frac{dZ}{dW}=-\frac{AZ}{nW-BZ},
\eeq
where 
\beq\label{ab}
	A=2(1+\textstyle\frac{\alpha}{1+V_*}),\qquad B=V_*(1+V_*)(\lambda+V_*).
\eeq
For later reference we note that $P_{+\infty}$ is a saddle point if and only if $A>0$.
The latter condition is satisfied when $\kappa=\hat\kappa$, since $\alpha$ 
then vanishes (see \eq{alpha} and \eq{isentr_kappa}).

\section{Absence of Guderley solutions when $0<\lambda<1$}\label{no_Guderley}
Recall from Section \ref{main_findings} that a {\em Guderley solution} 
refers to a radial self-similar solution of the form 
\eq{alt_sim_vars} to the Euler system 
\eq{m_eul}-\eq{ener_eul}, defined (at least)
for all negative times, and in which a 
converging shock wave approaches the origin by propagating 
into a quiescent fluid near the center of motion (i.e.,
the fluid is at rest and at constant pressure and density there). 

It is further assumed that the parameters $\kappa$ and $\lambda$
are the same inside and outside of the converging shock, and that the 
shock follows a path with $x=\frac{t}{r^\lambda}\equiv x_\s$, 
where $x_\s$ is a negative constant. As pointed out by Lazarus \cite{laz}, the 
constant density inside the converging shock implies that the
parameter $\kappa$ must be zero for a Guderley solution. 
In this work we assume $\lambda\neq 1$, and 
it follows from \eq{alt_sim_vars} that the sound speed $c$ must vanish within
the quiescent region $x<x_\s$. For the case of an ideal gas, this means that 
the temperature vanishes identically there.
(We stress that the density within the quiescent region does not 
vanish in a Guderley solution; the collapse of a spherical vacuum 
region is a different problem which also admits similarity solutions, 
\cites{bk,hun_63,laz}.) Therefore, for a Guderley solution, we have
$(V(x),C(x))\equiv (0,0)$ for $-\infty<x<x_\s$. 

To the best of our knowledge, starting with \cite{gud}, all works on
Guderley solutions assume that $\lambda\geq 1$. In order that the 
shock accelerates and collapses with infinite speed one must 
have $\lambda>1$. Among the many works on Guderley solutions 
(and collapsing cavities) we have found only a few that address 
the choice of range for $\lambda$. Among these, \cites{laz,hun_63}
simply choose to disregard cases where the shock
collapses with vanishing speed, while \cite{bk} (p.\ 16) claims that 
$\lambda<1$ ``is incompatible with any finite pressure before the wave.''
However, we do not see any {\em a priori} reason to exclude cases with
$0<\lambda<1$. If the Euler system admitted converging similarity shocks 
for this range, they would provide examples of ``glancing'' shocks that
weaken, slow down, and reach the origin with zero speed. 

However, based on numerical tests, we conjecture that the Euler system
(for an ideal gas in $2$ or $3$ space dimensions) does not admit 
Guderley solutions with this type of glancing similarity shock. 
In the rest of this section we briefly describe the analysis leading to 
this conclusion. Thus, for the remainder of this section, the assumptions 
\[0<\lambda<1\qquad\text{and}\qquad \kappa=0\]
are in force.
First, the Rankine-Hugoniot relations for a discontinuity propagating 
in a similarity solution along a curve $x\equiv x_\s$ in the $(r,t)$-plane are
\begin{align}
	1+V_+&=\textstyle\frac{\gamma-1}{\gamma+1}(1+V_-)+\frac{2C_-^2}{(\gamma+1)(1+V_-)}
	\label{rh1}\\
	C_+^2&=C_-^2+\textstyle\frac{\gamma-1}{2}[(1+V_-)^2-(1+V_+)^2]
	\label{rh2}\\
	R_+(1+V_+)&=R_-(1+V_-),
	\label{rh3}
\end{align}
where the subscripts $-$ and $+$ refer to states immediately prior 
to and after passing through the shock, respectively. Note that, in a Guderley 
solution the converging shock invades a quiescent state where $(V_-,C_-)=(0,0)$,  
and it follows from \eq{rh1}-\eq{rh2} that 
\beq\label{P_+}
	P_+=(V_+,C_+)
	=\big(-\textstyle\frac{2}{\gamma+1},\frac{\sqrt{2\gamma(\gamma-1)}}{\gamma+1}\big).
\eeq
Note that $P_+$ depends only on $\gamma$ and is located on the graph of 
the function $C_+(V):=\sqrt{(1+V)(2+V)}$ for $-1<V<0$.

Next, the entropy conditions for a converging similarity 1-shock defined for negatives times 
take the form
\beq\label{lax}
	C_-<1+V_-\qquad\text{and}\qquad C_+>1+V_+.
\eeq
(It may be shown from the entropy conditions that any converging self-similar 
shock with a quiescent inner state is necessarily a 1-shock when $\lambda>0$.) 
As is evident from \eq{P_+}, $P_+$ is located above the critical line $\mathcal L_+$, 
so that the 1-shock in a Guderley solution is entropy admissible.

To construct a Guderley solution (defined for all $t<0$) it is then necessary to find a solution 
$(V(x),C(x))$ of \eq{V_sim2}-\eq{C_sim2} which starts out from $P_+$ with $x=x_\s<0$,
and reaches the critical point $P_1$ at the origin with $x=0$. In particular, it must cross the critical 
line $\mathcal L_+$. Since the denominator $D(V,C)$ in \eq{V_sim2}-\eq{C_sim2} 
vanishes there, the only possibility is that the trajectory 
crosses at one of the critical points $P_6$ or $P_8$\footnote{See \eq{on_L_+}. 
There is an apparent 
third possibility in the exceptional case that $P_4$ happens to lie on $\mathcal L_+$.
However, this is not a separate case: it can be shown that if $P_4\in\mathcal L_+$, then 
$P_4$ necessarily coincides with either $P_6$ or $P_8$.}.
For fixed $n=2$ or $3$ and $\gamma>1$ it turns out that only certain values of $\lambda$
makes this happen.


This non-linear eigenvalue problem for $\lambda$ was first addressed
by Guderley \cite{gud}, and later by several authors, see \cite{rkb_12}. As far as we know, the most 
comprehensive treatment (always for $\lambda\geq 1$) is due to Lazarus \cite{laz}, who 
also carried out detailed numerical 
calculations. One conclusion of these works is that, for $n=2$ or $3$ and 
$\gamma>1$, there is always at least one $\lambda$-value $\lambda>1$ for which the 
trajectory starting at $P_+$  passes through either $P_6$ or $P_8$,
and then proceeds to reach $P_1=(0,0)$. (In fact, depending on $\gamma$,
there can be whole intervals of allowed $\lambda$-values; also, once $P_6$
or $P_8$ has been reached from $P_+$, there may be infinitely many trajectories 
connecting to the origin; see \cite{laz}.) 

With this background we now turn to the possibility of generating a Guderley 
solution when $0<\lambda<1$. As $\kappa=0$ it follows from
Section \ref{lam_max_min_kappa=0} that $V_6=V_-$ and $V_8=V_+$ are 
real, so that the critical points $P_6,P_8$ 
are necessarily present. A direct calculation using \eq{V_4} and \eq{V_pm} shows
that 
\[V_6<-1<-\lambda<V_4<0<V_*<V_8\]
in this case. (We omit the details; similar computations are detailed in 
the proof Lemma \ref{V_locns} below.) It follows that $P_6$ is located to the left 
of the vertical line $V=-1$, and therefore belongs to $\mathcal L_-$, while $P_8$
is located to the right of the vertical asymptote $V=V_*$ of $\mathcal G$, and  
lies on $\mathcal L_+$. Also, in the case under consideration, $0<C_4<1+V_4$, 
so that $P_4$ is located strictly below $\mathcal L_+$. 

Let $\Gamma_+$ denote the sought-for trajectory (i.e., staring at $P_+$
and ending at $P_1=(0,0)$). Since $x<0$ along $\Gamma_+$, and since its starting point $P_+$  
lies above $\mathcal L_+$, it follows from \eq{V_sim2}-\eq{C_sim2} 
that the trajectory moves in the direction of the vector field $(-G(V,C),-F(V,C))$ as $x$
increases from $x_\s<0$ toward $0$. It may be verified that $G(V,C)<0$ whenever
$(V,C)$ belongs to the region $\mathcal R:=\{-1<V<0\text{ and } C>1+V\}$. Since $\Gamma_+$ 
starts at $P_+\in\mathcal R$, it starts out moving to the right, and it follows that the only 
possibility for $\Gamma_+$ to reach $P_1$ is by crossing $\mathcal L_+$ at $P_8$.

An inspection of the ODE system  \eq{V_sim2}-\eq{C_sim2}  shows that, depending on the
value of $\gamma>1$, this could potentially happen in one 
of two ways: 
\begin{enumerate}
	\item[(A)] either $P_+$ is located near $P_2=(-1,0)$ and above $\mathcal F=\{(V,C)\,:\, F(V,C)=0\}$ (this happens for
	$\gamma$-values sufficiently close to $1$), and $\Gamma_+$ would start out by 
	moving up in a North-East direction, then cross $\mathcal F$ horizontally, before 
	moving down in a South-East direction toward $P_8$; or,
	\item[(B)] $P_+$ is located below $\mathcal F$, and $\Gamma_+$ would move 
	monotonically in a South-East direction toward $P_8$. (This could only occur for
	$\gamma$ sufficiently large, so that $P_+$ is located above $P_8$.)
\end{enumerate}
However, numerical tests with various choices for $\lambda\in(0,1)$ and $\gamma>1$
indicate that neither of these scenarios actually occurs. In all cases we have considered
the trajectory  $\Gamma_+$ hits $\mathcal L_+$ well to the left of $P_8$.
To have $\Gamma_+$ reach $P_8$ it appears advantageous to choose 
$\gamma\gg1$, so that $P_+\approx(0,\sqrt 2)$ is as close as possible to $P_8$. However, even with extreme
values for $\gamma$ (of order $10^6$, say), we have not been able to find cases where $\Gamma_+$ 
even crosses into the right half-plane (where $P_8$ is located) before hitting $\mathcal L_+$.

We therefore abandon the search for Guderley solutions when $\lambda\in(0,1)$, 
and instead turn to the construction of {\em shock-free} solutions for this parameter 
regime. To do so we first need to consider constraints imposed on the similarity
parameters $\lambda$, $\kappa$.

\section{Restrictions on $\lambda$ and $\kappa$}\label{lambda_restrict}\label{lam_kap_constrs}
In this section the similarity parameters $\lambda$ and $\kappa$ are 
at the outset free, while $\gamma>1$ is fixed and $n=2,3$. The goal is to obtain 
restrictions on $\lambda$ and $\kappa$ from physically relevant constraints
as described below. Some of the arguments in this section are similar to 
those in \cite{jt2}; for completeness we include the details.

\subsection{Restrictions from integral bounds}\label{int_bounds}
Referring to the discussion in Section \ref{P1_P3} we restrict attention to 
solutions $(V(x),C(x))$ of \eq{V_sim2}-\eq{C_sim2} which  
pass through the origin with \eq{well_bhvd} satisfied.
It follows from \eq{alt_sim_vars} that the flow variables at time of collapse are given by
\beq\label{at_collapse}
	\rho(0,r)=R(0)r^\kappa\qquad 
	u(0,r)=-\textstyle\frac{\nu}{\lambda}r^{1-\lambda}, \qquad 
	c(0,r)=-\frac{\mu}{\lambda}r^{1-\lambda}.
\eeq
As a minimal, physical requirement we insist that the resulting flow has locally 
finite mass, momentum, and total energy, i.e., for each ${\bar r}>0$, we have
\[\int_0^{\bar r}\rho(t,r)r^m\, dr,\quad
\int_0^{\bar r}\rho(t,r)|u(t,r)|r^m\, dr,\quad
\int_0^{\bar r}\rho(t,r)\left(e(t,r)+\textstyle\frac{1}{2}|u(t,r)|^2\right)r^m\, dr<\infty.
\]
Using \eq{at_collapse} it is straightforward to verify that,
at time $t=0$, these  integral bounds imply
\begin{itemize}
	\item[(I)] $\kappa+n>0$
	\item[(II)] $\lambda<1+\kappa+n$
	\item[(III)] $\lambda<1+\textstyle\frac{\kappa+n}{2}$,
\end{itemize}
respectively. Note that (II) is a consequence of (I) and (III).
For later reference we record the following consequence:
According to (III) and the standing assumption $\gamma>1$, we have
\beq\label{1+V*}
	0<\textstyle\frac{n+\kappa-2(\lambda-1)}{n\gamma}
	<\textstyle\frac{n\gamma+\kappa-2(\lambda-1)}{n\gamma}=1+V_*.
\eeq

\subsection{Restrictions from pointwise bounds in a continuous flow}\label{cont_restrcn}
The restrictions (I)-(III) above are now in force; in particular, \eq{1+V*} holds.
We then consider any solution $(V(x),C(x))$ of the similarity ODEs 
\eq{V_sim2}-\eq{C_sim2} which is defined for all $x<0$, and with the property 
that it defines a {\em continuous} Euler flow for all $t<0$.
As far as we are aware, the only way for this to occur
is by having the solution $(V(x),C(x))$ approach the critical point $P_{+\infty}$
in the upper half-plane:
\beq\label{behav}
	(V(x),C(x))\to P_{+\infty}=(V_*,+\infty)
	\qquad\text{as $x\downarrow -\infty$.}
\eeq
The latter property will hold, by construction, for the continuous solutions we analyze 
in Section \ref{cont_constrcn}, and \eq{behav} is assumed for the remainder of the 
present section. 

\begin{remark}\label{vac_regions}
	Strictly speaking, there may be another type of continuous 
	similarity flows with $0<\lambda<1$ violating \eq{behav}, viz.\ flows
	describing a spherical cavity (vacuum region) being filled by an inflowing gas. 
	In this work we restrict attention to flows without open vacuum regions
	(but see Remark \ref{one_point_vac}).
\end{remark}
By imposing continuity of the flow for $t<0$, we require
that the primary flow variables $\rho$, $u$, and $c$ are 
locally bounded at any fixed time strictly prior to collapse. In particular, 
$\rho(\bar t,r)$, $u(\bar t,r)$, and $c(\bar t,r)$ should remain bounded 
as $r\downarrow 0$ whenever $\bar t<0$. We proceed to analyze the 
implications of these requirements.
For $\bar t<0$ fixed we have 
\[u(\bar t,r)=-\textstyle\frac{r^{1-\lambda}}{\lambda}\frac{V(x)}{x}
=-\frac{1}{\lambda \bar t}V(x)r\propto V(x)r. \]
From \eq{behav} it follows that $u(\bar t,r)\sim r$ as $r\downarrow 0$. 
This shows that the speed of the fluid particles, at any time $\bar t<0$, approach zero at a 
linear rate as the center of motion is approached. Thus, no additional 
constraint is imposed on the similarity parameters $\lambda$ and $\kappa$ 
by requiring bounded (indeed, vanishing) 
fluid speed at the center of motion. 

Next, to analyze $c(\bar t,r)$ as $r\downarrow 0$, we need the leading 
order behavior of $C(x)$ as $x\downarrow -\infty$. Applying \eq{behav} 
in \eq{C_sim2} gives
\[\textstyle\frac{1}{C}\frac{d C}{dx}\sim \frac{1}{\lambda}
(1+\frac{\alpha}{1+V_*}) \frac{1}{x} \qquad\text{as $x\downarrow -\infty$,}\]
so that
\beq\label{sigma}
	C(x)\sim |x|^\sigma\qquad\text{as $x\downarrow -\infty$, where} \qquad
	\sigma=\textstyle\frac{1}{\lambda}(1+\frac{\alpha}{1+V_*}).
\eeq
As $\bar t$ is fixed, we have $x\propto- r^{-\lambda}$ and \eq{alt_sim_vars} gives
\beq\label{c_bar}
	c(\bar t,r)\sim r^{1-\sigma\lambda}\qquad\text{as $r\downarrow 0$.}
\eeq
Boundedness of $c(\bar t,r)$ as $r\downarrow 0$ therefore imposes the constraint 
$1-\sigma\lambda\geq0$. According to \eq{sigma} and \eq{1+V*}, this 
amounts to $\alpha\leq 0$, or, according to \eq{alpha},
\beq\label{alfa}
	2(\lambda-1)+\kappa(\gamma-1)\leq 0.
\eeq
Next, to obtain the behavior of $\rho(\bar t,r)$ as $r\downarrow 0$, we use 
the exact integral \eq{entr_int} together with 
$V(x)\sim V_*$, $C(x)\sim |x|^\sigma$, and $x\propto r^{-\lambda}$, 
to get that
\beq\label{rho_bar}
	\rho(\bar t,r)\sim r^{\kappa+\frac{2\lambda(\sigma-1)}{1-\gamma+q}}
	\qquad\text{as $r\downarrow 0$,}
\eeq
where $q$ is given by \eq{q}. Boundedness of $\rho(\bar t,r)$ as $r\downarrow 0$
therefore requires
\beq\label{q_etc}
	\kappa+\textstyle\frac{2\lambda(\sigma-1)}{1-\gamma+q}\geq 0.
\eeq
We claim that \eq{q_etc}, together with requirement (I) in Section 
\ref{int_bounds}, \eq{1+V*}, and \eq{alfa}, 
imply that $\kappa$ must take the ``isentropic'' value $\hat\kappa$ given in \eq{isentr_kappa}. 
To see this, note that \eq{q}, (I), and \eq{alfa} (i.e., $\alpha\leq 0$) give $q\leq 0$. 
Therefore, the denominator in \eq{q_etc} satisfies $1-\gamma+q<0$, and 
\eq{q_etc} holds if and only if
\[\kappa(1-\gamma+q)+2\lambda(\sigma-1)\leq 0.\]
Using \eq{q} and \eq{sigma} to substitute for $q$ and $\sigma$, and 
rearranging, we obtain the equivalent condition 
\beq\label{almost_there}
	\textstyle\frac{\alpha}{1+V_*}\leq 
	\frac{n}{\kappa+n}[(\lambda-1)+\frac{\kappa}{2}(\gamma-1)]
	\equiv\frac{n\gamma \alpha}{\kappa+n}.
\eeq
Recall that boundedness of $c(\bar t,r)$ near $r=0$ requires 
\eq{alfa}, i.e., $\alpha\leq 0$. If $\alpha<0$ \eq{almost_there} simplifies to
\[\textstyle\frac{1}{1+V_*}\geq \frac{n\gamma}{\kappa+n},\]
which, according to \eq{1+V*}, (I), and \eq{V_*}, reduces to
\beq\label{almost_almost_there}
	n(\gamma-1)\leq 2(\lambda-1).
\eeq
However, $\alpha<0$ also gives $2(\lambda-1)<-\kappa(\gamma-1)$, 
so that \eq{almost_almost_there} yields $n(\gamma-1)<-\kappa(\gamma-1)$, 
or $n+\kappa<0$. This contradicts the integrability condition (I), and we 
conclude that $\alpha$ must vanish, i.e., we must have $\kappa=\hat\kappa$.
We observe that, with $\kappa=\hat\kappa$, \eq{rho_bar} and \eq{c_bar}
indeed provide bounded values for both $\rho(\bar t,r)$ and $c(\bar t,r)$ as $r\downarrow0$.
As detailed above (after \eq{isentr_kappa}), it follows that the resulting
flow in this case is globally isentropic.
We sum up our findings in the following proposition:
\begin{proposition}\label{isentr_cont_flow}
	Let $n=2$ or $3$ and fix $\gamma>1$
	and $\lambda>0$. Consider any solution $(V(x),C(x))$  
	of the similarity ODEs \eq{V_sim2}-\eq{C_sim2}, defined for $x<0$
	and satisfying \eq{behav}
	and \eq{well_bhvd} (with $\mu$ and $\nu$ finite and nonzero). Finally, 
	let $R\geq 0$ be given by \eq{entr_int}, and define the flow variables
	$\rho$, $u$, $c$ according to \eq{alt_sim_vars}.
	
	Then the requirements {\em (I)-(III)} in Section \ref{int_bounds}, together
	with boundedness of $\rho(t,r)$ and $c(t,r)$ as $r\downarrow0$ at fixed 
	times $t<0$, imply that the similarity parameter $\kappa$ in
	\eq{alt_sim_vars} must have the value $\hat\kappa$ given in \eq{isentr_kappa}.
	Finally, with $\kappa=\hat\kappa$ the resulting Euler flow (defined for $t<0$)
	is necessarily isentropic. 
\end{proposition}
From now on $\kappa=\hat\kappa$ is assumed. We note that the integrability 
conditions (I)-(III) in Section \ref{int_bounds} then reduce to the single 
requirement (III), which now reads
\beq\label{1st_lam_constr}
	\lambda<\bar\lambda(\gamma,n):=1+\textstyle\frac{n}{2}(1-\frac{1}{\gamma}).
\eeq
This is trivially satisfied if $0<\lambda<1$, a fact we make use of in Section 
\ref{cont_constrcn}.

\subsection{Isentropic behavior near $r=0$}\label{r=0_behavior}
The arguments above show that with $\kappa=\hat\kappa$, any continuous 
solution of \eq{V_sim2}-\eq{C_sim2} which is defined for all $x<0$ and
satisfies \eq{behav} and \eq{well_bhvd}, generates flow variables 
$u(\bar t,r)$, $c(\bar t,r)$, $\rho(\bar t,r)$ that approach finite values as 
$r\downarrow 0$ at each fixed $\bar t<0$.
We now verify that these finite values are approached with {\em bounded 
gradients}. In particular, no gradient catastrophe occurs in the flow prior 
to collapse at time $t=0$.

First, consider $u_r(\bar t,r)$; according to \eq{alt_sim_vars} and 
\eq{V_sim2} we have
\[u_r(\bar t,r)=-\textstyle\frac{1}{\lambda \bar t}\left(V(x)+\frac{G(V(x),C(x))}{D(V(x),C(x))}\right),
\qquad\text{where $x=\frac{\bar t}{r^\lambda}.$}\]
Recalling that $C(x)\uparrow+\infty$ and $V(x)\to V_*$ as $r\downarrow 0$, we get from 
\eq{G} that 
\[\textstyle\frac{G(V(x),C(x))}{D(V(x),C(x))}\to 0 \qquad\text{as $r\downarrow 0$.}\]
It follows from this that 
\[u_r(\bar t,r)\sim -\textstyle\frac{V_*}{\lambda \bar t}\qquad\text{as $r\downarrow 0$.}\]
Similarly, using \eq{alt_sim_vars} and \eq{C_sim2}, we have
\[c_r(\bar t,r)=\textstyle\frac{C(x)}{\lambda \bar t}
\left[\frac{(1-k_1)(1+V(x))^2+k_2(1+V(x))-k_3}{C(x)^2-(1+V(x))^2}\right],\]
and it follows from \eq{behav} that $c_r(\bar t,r)\sim 0$ as $r\downarrow 0$. In particular, to leading 
order, $c(\bar t,r)$ is constant as $r\downarrow 0$. Finally, since 
$\rho\propto c^\frac{2}{\gamma-1}$ in isentropic flow, the same applies to the density field.

We conclude that, in the isentropic setting under consideration, at any fixed 
time $\bar t<0$, all of $u_r(\bar t,r)$, $c_r(\bar t,r)$, $\rho_r(\bar t,r)$, and 
hence also $p_r(\bar t,r)$, remain bounded as $r\downarrow 0$. In particular,
no gradient catastrophe occurs at $r=0$ at strictly negative times. 

On the other hand, {\em at} time of collapse $t=0$, \eq{well_bhvd} and \eq{alt_sim_vars} 
give
\[\rho(0,r)=r^{\hat \kappa} R(0),\qquad u(0,r)=-\textstyle\frac{\nu}{\lambda}r^{1-\lambda},
\qquad c(0,r)=-\textstyle\frac{\omega}{\lambda}r^{1-\lambda}.\]
In particular, provided $\nu$ and $\omega$ are finite and nonzero, and $0<\lambda<1$,
we see that both the velocity and sound speed suffer a gradient catastrophe at the origin
at $t=0$. 

The same applies to the density field provided $\hat \kappa<1$, i.e.,
$\lambda>\frac{3-\gamma}{2}$. Note that the latter inequality is
satisfied whenever $\lambda>0$ and $\gamma>3$, as will be the case for 
the solutions we construct in Section \ref{cont_constrcn}.
On the other hand, the pressure field at time of collapse is, by \eq{sound_speed},
\[p(0,r)=\textstyle\frac{1}{\gamma}\rho(0,r)c^2(0,r)
\propto r^{\hat \kappa+2(1-\lambda)},\]
which suffers a gradient catastrophe at $r=0$ provided $\hat \kappa+2(1-\lambda)<1$,
or equivalently,
\beq\label{grad_catas_p}
	\lambda> \textstyle\frac{1}{2}(1+\frac{1}{\gamma}).
\eeq
As we shall see, \eq{grad_catas_p} will be violated for all solutions we 
construct below: their pressure fields are at least $C^1$-smooth at 
time of collapse.

\begin{remark}\label{one_point_vac}
	We note that, with $\lambda\in(0,1)$ and $\kappa:=\hat\kappa$ the density
	field at time of collapse satisfies $\rho(0,r)\propto r^{\hat\kappa}$, which vanishes at the 
	origin. The resulting Euler flow therefore has a {\em one-point vacuum}
	at the origin at time of collapse.
\end{remark}

\section{Construction of continuous flows with 
$0<\lambda<1$ and $\kappa=\hat\kappa$}\label{cont_constrcn}
We now turn to the construction of continuous, and in particular, locally bounded
radial Euler flows with similarity variable $\lambda\in(0,1)$. As explained in 
Section \ref{cont_restrcn}, we restrict attention to solutions satisfying \eq{behav}, 
and Proposition \ref{isentr_cont_flow} then shows that we must choose 
$\kappa=\hat\kappa$ in order to meet the physical constraints (I)-(III) in 
Section \ref{int_bounds}. Thus, for the remainder of the paper it is assumed that 
\beq\label{isntr_0<lam<1}
	0<\lambda<1\qquad\text{and}\qquad \kappa=\hat\kappa
	=\textstyle\frac{2(1-\lambda)}{\gamma-1}.
\eeq
We observed at the end of Section \ref{isntr_sing_pnts} that $V_6=V_-$ 
and $V_8=V_+$ are both real under assumptions \eq{isntr_0<lam<1}, so that 
the critical points $P_1$-$P_9$ are all present.

\subsection{Outline of construction}\label{outline}
The continuous flows are built by identifying solution 
trajectories $\Gamma_1$-$\Gamma_3$ of \eq{V_sim2}-\eq{C_sim2} 
with the properties 
\begin{itemize}
	\item[($\Pi_1$)] $\Gamma_1$ connects $P_{+\infty}$ to $P_8$;
	\item[($\Pi_2$)]  $\Gamma_2$ connects $P_8$ to $P_9$ and 
	passes through $P_1=(0,0)$;
	\item[($\Pi_3$)]  $\Gamma_3$ connects $P_9$ to $P_{-\infty}$.
\end{itemize}
The symmetries recorded in \eq{symms} effectively reduce our task to 
identifying only $\Gamma_1$ and $\Gamma_2$: $\Gamma_3$
will simply be the the reflection of $\Gamma_1$ about the $V$-axis,
connecting $P_9$ to $P_{-\infty}$. 

Recall from Section \ref{P_infs} that $P_{\pm\infty}$ are saddle points.
In searching for a trajectory $\Gamma_1$ satisfying ($\Pi_1$) it is therefore
advantageous that $P_8$ be a nodal point into which a large family of 
trajectories are drawn. Indeed, a key part of the following analysis 
(Section \ref{P_8_analysis}) concerns the identification of a 
$(\lambda,\gamma)$-regime for which $P_8$ is a proper node. 

Care must be taken that $\Gamma_1$ reaches $P_8$ 
without crossing
the critical line $\mathcal L_+$. It turns out that
the latter requirement fixes the spatial dimension to be $n=3$
(see Section \ref{L_1_neg}). Likewise, the trajectory $\Gamma_2$
must connect to the origin without crossing $\mathcal L_+$; this 
however, will not impose any further constraints on the parameters.
 
Having identified a suitable $(\lambda,\gamma)$-regime (for $n=3$) 
we find it necessary to verify numerically that there are 
cases in which $\Gamma_1$ connects $P_{+\infty}$ to $P_8$ without 
first crossing $\mathcal L_+$. With $\Gamma_1$, and hence $\Gamma_3$,
thus determined, it remains to determine a suitable trajectory $\Gamma_2$
satisfying ($\Pi_2$). It is unproblematic to reach the origin $P_1$ from $P_8$: 
among the trajectories leaving the node at $P_8$, there are infinitely many
that connect to the origin. However, as the trajectory 
is continued through the origin, it should subsequently be drawn into $P_9$. 
It turns out that this last requirement determines a range of possible slopes 
with which $\Gamma_2$ can reach $P_1$;
see Section \ref{P_1_analysis}. Again, we verify numerically the existence of
trajectories $\Gamma_2$ meeting these constraints.

We note that it is necessary to make a final check on the selected trajectories 
$\Gamma_1$-$\Gamma_3$: they must provide admissible solutions trajectories 
for the original ODE system \eq{V_sim2}-\eq{C_sim2} (as opposed to the 
single ODE \eq{CV_ode}). As explained 
in Section \ref{P1_P3}, any solution $(V(x),C(x))$ of physical relevance
must necessarily pass through the origin $P_1$ with $x=0$, and from the upper 
half-plane to the lower half-plane as $x$ increases. Also, they must move along 
$\Pi_1$-$\Pi_2$-$\Pi_3$ in the correct direction given by \eq{V_sim2}-\eq{C_sim2} 
as $x$ increases from $-\infty$ to $+\infty$. 

Figure \ref{Figure_2} provides a representative case of the vector field 
$(-\frac{1}{\lambda x}\frac{G(V,C)}{D(V,C)},-\frac{1}{\lambda x}\frac{F(V,C)}{D(V,C)})$
corresponding to the ODE system \eq{V_sim2}-\eq{C_sim2}, with $x<0$ ($x>0$) in 
the upper (lower) half-plane. Notice that the arrows provide the actual 
direction of flow for solutions $(V(x),C(x))$ as $x$ increases.

\begin{figure}
	\centering
	\includegraphics[width=8cm,height=8cm]{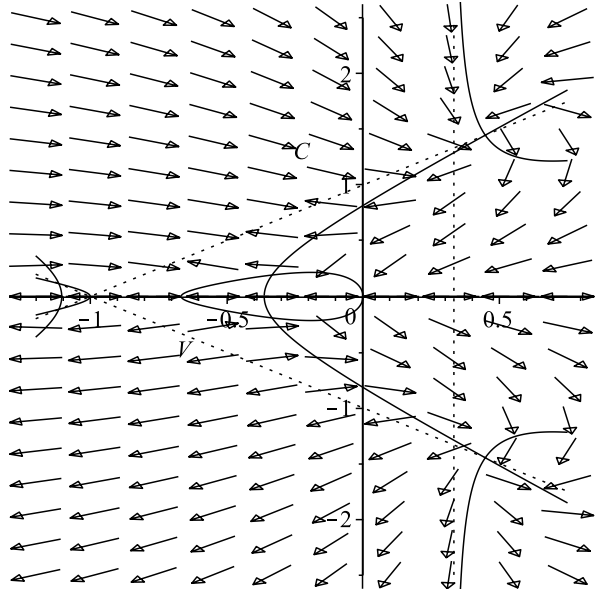}
	\caption{The direction field of \eq{V_sim2}-\eq{C_sim2}, with $x<0$ ($x>0$) in 
	the upper (lower) half-plane. The zero-levels of $F(V,C)$, $G(V,C)$, 
	the critical lines $\mathcal L_\pm=\{C=\pm(1+V)\}$,
	and the axes are solid curves; the vertical asymptote 
	$V=V_*$ is dotted. The parameters are as in Figure \ref{crit_points}: $n=3$, $\gamma=\frac{5}{3}$,
	$\lambda=\frac{2}{3}$, and $\kappa=\hat\kappa=1$. }\label{Figure_2}
\end{figure}

\begin{remark}
	Figure \ref{Figure_2} illustrates the impossibility of having a solution of the 
	ODE system \eq{V_sim2}-\eq{C_sim2} cross the critical lines 
	$\mathcal L_\pm=\{C=\pm(1+V)\}$ at a 
	non-singular point: such trajectories of the autonomous ODE
	\eq{CV_ode} fail to yield relevant
	solutions to \eq{V_sim2}-\eq{C_sim2} since the vector field 
	corresponding to \eq{V_sim2}-\eq{C_sim2} points in opposite directions on 
	either side of $\mathcal L_\pm$. 
	
	Note that the parameters in Figure \ref{Figure_2} are chosen for illustrative
	purposes; in particular, the behavior near $P_8$ is such that property 
	{\em ($\Pi_1$)} fails in this case. It will be shown below that to 
	satisfy {\em ($\Pi_1$)}, we need to choose $n=3$, $\lambda$ sufficiently small,
	and $\gamma$ sufficiently large. It turns out that with such parameter values it is 
	necessary to zoom in at the critical points $P_8$, $P_1$, and $P_9$ in order to 
	display the behavior there; see Figures \ref{Zoom_Gamma_2} and \ref{Zoom_Hugoniot}.
\end{remark}

\subsection{Location of critical points}
We start by determining the relative $V$-locations of the  
critical points under the assumptions in \eq{isntr_0<lam<1}. 
For convenience we repeat the expressions for 
$V_\pm$ (see \eq{isntr_V_pm} and \eq{aQ}) and $V_*$ in terms of 
$\mu=\lambda-1$:
\beq\label{isntr_0<lam<1_V_pm}
	V_\pm=\textstyle\frac{1}{2}(a\pm\sqrt{Q})\qquad\text{and}\qquad
	V_*=\frac{-2\mu}{n(\gamma-1)},
\eeq
where 
\beq\label{aQ_2}
	a=\textstyle\frac{(\gamma-3)}{m(\gamma-1)}\mu-1\qquad\qquad
	Q=\left(\textstyle\frac{(\gamma-3)}{m(\gamma-1)}\right)^2\!\!\mu^2
	-2\textstyle\frac{(\gamma+1)}{m(\gamma-1)}\mu+1.
\eeq
Recalling \eq{V_4} and introducing the positive constant
\beq\label{k}
	k:=\textstyle\frac{m}{2}[n(\gamma-1)+2],
\eeq
we have 
\beq\label{V_4_alt}
	V_4=-\textstyle\frac{m}{k}(1+\mu).
\eeq
For later reference we note that 
\beq\label{V8-V6}
V_8-V_6\equiv V_+-V_-=\sqrt{Q},
\eeq
and
\beq\label{V8-V4}
V_8-V_4=\textstyle\frac{1}{2}\left[ \left( \frac{(\gamma-3)}{m(\gamma-1)}+\frac{2m}{k} \right)\mu
+\left( \frac{2m}{k} -1\right)+\sqrt{Q}. \right]
\eeq
\begin{lemma}\label{V_locns}
Assuming $n=2$ or $3$, $\gamma>1$, and that \eq{isntr_0<lam<1} holds, we have
\beq\label{V_ineqs}
	V_6<V_2=-1<V_3=-\lambda<V_4<0<V_*<V_8.
\eeq
\end{lemma}
\begin{proof}
	We consider each inequality in turn, from left to right:
        \begin{itemize}
		\item $V_6=V_-<-1$: According to \eq{isntr_0<lam<1_V_pm}${}_1$ this inequality 
		amounts to 
		\beq\label{rootq1}
            		a+2<\sqrt{Q}.
        		\eeq
		It is immediate to verify that $a+2>0$ if and only if $(\gamma-1)(\mu+m)>2\mu$,
		which holds since $m\geq1$, $-1<\mu<0$, and $\gamma>1$. It follows that \eq{rootq1}
		is equivalent to $(a+2)^2<Q$; substituting from \eq{aQ_2} shows that the latter inequality 
		reduces to $\gamma>1$, establishing the first inequality.
		\item $-1<-\lambda$: Immediate by \eq{isntr_0<lam<1}${}_1$.
		\item $-\lambda<V_4$: Substituting from \eq{k} and \eq{V_4_alt}, and recalling that $\mu+1=\lambda>0$,
		this inequality reduces to $n(\gamma-1)>0$, which holds since $\gamma>1$.
		\item $V_4<0$: Immediate from \eq{V_4_alt} since $m$, $k$ (see \eq{k}), and $\mu+1=\lambda$ are all positive.
		\item $0<V_*$: Immediate from \eq{isntr_0<lam<1_V_pm}${}_2$ since $\mu<0<\gamma-1$.
		\item $V_*<V_8=V_+$: By substituting from \eq{isntr_0<lam<1_V_pm} we obtain the equivalent 
		inequality
        		\beq\label{rootq2}
            		-\textstyle\frac{4\mu}{n(\gamma-1)}-a<\sqrt{Q}.
       		\eeq
		Using the expression for $a$ in \eq{aQ_2} and rearranging, we get that the left-hand side of \eq{rootq2}
		is positive provided $n(\gamma+1)-4>\frac{mn(\gamma-1)}{\mu}$, which holds since the left-hand side 
		in the latter inequality is positive (because $n\geq 2$ and $\gamma>1$), while the right-hand side is negative. 
		It follows that \eq{rootq2} is equivalent to
		\beq\label{rootq3}
            		\big(\textstyle\frac{4\mu}{n(\gamma-1)}+a\big)^2<Q.
       		\eeq
		Substituting from \eq{aQ_2} for $a$ and $Q$, and simplifying the result, we obtain that
		\eq{rootq3} is equivalent to $n(\gamma-1)>-\mu(n(\gamma-1)-1)$, which holds since $\gamma>1$
		and $-\mu=1-\lambda<1$.
	\end{itemize}
\end{proof}

\subsection{The critical point $P_{+\infty}$}\label{Gamma_1}
This analysis was done in Section \ref{P_infs}, and it was noted there 
that $P_{+\infty}=(V_*,+\infty)$ is necessarily a saddle point when $\kappa=\hat \kappa$. 
It follows that there is a unique trajectory $\Gamma_1$ of \eq{CV_ode} 
which approaches $P_{+\infty}$.
An inspection of $F(V,C)$ and $G(V,C)$ shows that the solutions of
\eq{CV_ode} have negative slopes within the the region
\beq\label{Omega}
	\Omega:=\{(V,C)\,|\, V_*<V<V_8,\, 1+V<C<\bar C(V)\}
\eeq
where $C=\bar C(V)$ denotes the $V$-parametrization of the zero-level
$\mathcal G$ of $G$. Furthermore, their slopes are finite along $V=V_*$
and infinite along $C=\bar C(V)$. It follows that $\Gamma_1$ is located 
within $\Omega$ and can reach its boundary only along $C=1+V$ for some $V\in(V_*,V_8)$,
or at $P_8$. In order to be useful for our purpose of building a globally defined
fluid flow, we must have that $\Gamma_1$ passes through $P_8$. We proceed to 
analyze the behavior of \eq{CV_ode} around $P_8$.

\subsection{Behavior near $P_8$; construction of $\Gamma_1$}\label{P_8_analysis}
In this subsection, unless indicated differently, 
all quantities are evaluated at $P_8=(V_8,C_8)$, 
and the subscript `$8$' is suppressed in most of the expressions. 
To determine the type of the critical point $P_8$ we shall need the signs of various
quantities given in terms of the partial derivatives of $F$ and $G$ there.
First, since $P_8\in \mathcal L_+\cap\mathcal F\cap\mathcal G$ we have
\begin{align}
	C&=1+V\label{C_V_8_1}\\
	C^2&=k_1(1+V)^2-k_2(1+V)+k_3\label{C_V_8_2}\\
	C^2&=\textstyle\frac{V(1+V)(\lambda+V)}{n(V-V_*)}.\label{C_V_8_3}
\end{align}
At $P_8$ we then have
\begin{align}
	F_C&=2C^2\label{F_C_8}\\
	F_V&=C(k_2-2k_1(1+V))\label{F_V_8}\\
	G_C&=2nC(V-V_*)\equiv 2V(\lambda+V)\label{G_C_8}\\
	G_V&=C(n-\lambda+nV_*-2V).\label{G_V_8}
\end{align}
Here, $F_C$ and $F_V$ are calculated from \eq{F} (using that $\alpha=0$ 
when $\kappa=\hat\kappa$), 
while $G_C$ and $G_V$ are calculated from \eq{G}, and using that
\[nC(V-V_*)=V(\lambda+V),\]
the latter being a consequence of \eq{C_V_8_1} and \eq{C_V_8_3}. 
We note that \eq{F_C_8}, \eq{G_C_8}, and \eq{V_ineqs} give
\beq\label{F_C_and_G_C_8}
	F_C>0\qquad\text{and}\qquad  G_C>0.
\eeq
Next, using the expressions above, we obtain that
\beq\label{G_V+G_C_8_1}
	G_V+G_C=C(2mV+m-\mu-nV_*),
\eeq
and substitution of the expressions in \eq{isntr_0<lam<1_V_pm} for $V=V_8=V_+$ and $V_*$
then yields
\beq\label{G_V+G_C_8_2}
	G_V+G_C=mC\sqrt{Q}>0.
\eeq
Applying \eq{non_obvious_reln}, we therefore obtain
\beq\label{F_V+F_C_8}
	F_V+F_C=-\textstyle\frac{(\gamma-1)}{2}(G_V+G_C)<0,
\eeq
so that
\beq\label{F_V_F_C}
	F_V<-F_C<0.
\eeq
Finally, we note that the expressions above give
\beq\label{F_C+G_V_8}
	F_C+G_V=C\left(n+1-\textstyle\frac{\gamma+1}{\gamma-1}\mu\right)>0.
\eeq
We next recall some notation and results from Lazarus \cite{laz}. 
The Wronskian is defined by
\[W:=F_CG_V-F_VG_C,\]
and the discriminant by
\beq\label{R_sqrd}
	R^2:=(F_C-G_V)^2+4F_VG_C\equiv (F_C+G_V)^2-4W.
\eeq
In the following, whenever $R^2>0$, we set $R:=+\sqrt{R^2}>0$.
Next, with
\beq\label{L_12}
	L_{1,2}=\textstyle\frac{1}{2G_C}(F_C-G_V\pm R)
\eeq
and
\beq\label{E_12}
	E_{1,2}=\textstyle\frac{1}{2G_C}(F_C+G_V\pm R),
\eeq
and signs chosen so that 
\beq\label{Es}
	|E_1|<|E_2|, 
\eeq
we have that 
integrals of \eq{CV_ode} near $P_8$ 
approach one of the curves
\[(c-L_1v)^{E_1}=\text{constant}\times (c-L_2v)^{E_2},\]
where $v=V-V_8$ and $c=C-C_8$. Note that the signs $\pm$ in \eq{L_12} and
in \eq{E_12} agree; $L_1$ and $L_2$ are referred to as
the {\em primary} and {\em secondary} slopes (or directions), respectively. 
Provided that $R^2>0$ (so that $R$ is real and positive) and $W>0$, $P_8$ 
is a proper node. In this case all solution curves approaching $P_8$ do so with slope equal 
to the primary slope $L_1$, except one which approaches $P_8$ with slope $L_2$. 

An elegant argument by Lazarus \cite{laz} shows that $W\equiv W_8$ is given as
\beq\label{W_8}
	W=2kC_8^2(V_8-V_4)(V_8-V_6),
\eeq
where $k>0$ is given in \eq{k}. It follows from Lemma \ref{V_locns} that 
\beq\label{W_8_sign}
	W>0.
\eeq

We now assume that $R^2>0$ (this requirement is addressed below in
Section \ref{R^2>0}) and proceed to determine the signs to be used
in \eq{L_12} so that \eq{Es} is satisfied. 
According to \eq{E_12}, \eq{Es} holds if and only if
\beq\label{primary_ineq}
	|F_C+G_V\pm R|<|F_C+G_V\mp R|. 
\eeq
From \eq{F_C+G_V_8} we have that $F_C+G_V>0$, and since $W>0$ we have 
\beq\label{R_F_C_G_V_8}
	0<R=\sqrt{(F_C+G_V)^2-4W}<F_C+G_V. 
\eeq
It follows that the minus-sign should 
be used on the left hand side of \eq{primary_ineq}, and the plus-sign should 
be used on the right hand side of \eq{primary_ineq}. That is, under the condition
that $R^2>0$, together with the standing assumption \eq{isntr_0<lam<1}, we have
\beq\label{prmr_scdr_slopes_8}
	L_1=\textstyle\frac{1}{2G_C}(F_C-G_V-R)
	\qquad\text{and}\qquad
	L_2=\textstyle\frac{1}{2G_C}(F_C-G_V+R).
\eeq
With this we have that $P_8$ is a proper node and that all but one of the integrals 
of \eq{CV_ode} approaching $P_8$ do so with slope $L_1$. 

We next want to 
determine how the primary slope $L_1$ compares to those of the curves $\mathcal G$ and 
$\mathcal L_+$ at $P_8$. Let $\bar C(V)$ be the $V$-parametrization of the zero-level 
curve $\mathcal G$ for $G(V,C)$, so that $\bar C'(V_8)=-\frac{G_V}{G_C}$. Together 
with \eq{prmr_scdr_slopes_8}${}_1$, and the fact that $G_C>0$ (by 
\eq{F_C_and_G_C_8}${}_2$), this implies that the inequality $L_1>\bar C'(V_8)$ 
is equivalent to $F_C+G_V>R$, which holds according to \eq{R_F_C_G_V_8}.
Therefore, near $P_8$, the straight line 
\[\mathcal L_1:\qquad C=C_8+L_1(V-V_8)\]
is located below $\mathcal G$ for $V<V_8$ and above $\mathcal G$ for $V>V_8$.

Before proceeding we also note the following: with $\tilde C(V)$ denoting the 
$V$-parametrization of the zero-level curve $\mathcal F$ for $F(V,C)$, 
we have $\tilde C'(V)=-\frac{F_V}{F_C}$, and it follows from \eq{F_C_and_G_C_8}${}_1$
and \eq{F_V+F_C_8} that $\tilde C'(V_8)>1$. Similarly,  using \eq{F_C_and_G_C_8}${}_2$
and \eq{G_V+G_C_8_2}, we have $\bar C'(V_8)<1$.

Finally, since $\frac{dC}{dV}=\frac{F}{G}<0$ within the region $\Omega$
given in \eq{Omega}, a {\em necessary} condition for having $\Gamma_1$ approach $P_8$
is that $L_1<0$. As $G_C>0$, \eq{prmr_scdr_slopes_8}${}_1$ shows that this
condition amounts to
\beq\label{L_1<0}
	F_C<R+G_V.
\eeq

Our goal now is to identify a parameter regime $(\lambda,\gamma)\in (0,1)\times (1,\infty)$
for which both of the two requirements $R^2>0$ and $L_1<0$ 
are satisfied. As demonstrated below in Lemma \ref{L_1_neg_lemma}, this requires $n=3$,
$0<\lambda<\frac{1}{9}$, and $\gamma$ sufficiently large. 
For such parameter values we then verify numerically that there are cases 
in which 
\begin{itemize}
	\item there are integrals of \eq{CV_ode} passing through the node $P_8$  
	along the primary direction $L_1$ and crossing the vertical line $V=V_*$; and 
	\item there are other integrals of \eq{CV_ode} passing through the node $P_8$  
	along the primary direction $L_1$ and crossing $\mathcal G$ vertically. 
\end{itemize}
It then follows by continuity that the unique integral $\Gamma_1$ of \eq{CV_ode} which approaches the 
critical point $P_{+\infty}=(V_*,+\infty)$, also passes through $P_8$
along the primary direction $L_1$.

\subsubsection{The requirement $R^2(\lambda,\gamma)>0$.}\label{R^2>0}
The discriminant $R^2$ is given in \eq{R_sqrd}; fixing $n=2$ or $3$, it is a 
function $R^2(\lambda,\gamma)$. According to \eq{F_C+G_V_8}, \eq{R_sqrd}, and \eq{W_8}, 
we have $R^2(\lambda,\gamma)>0$ if and only if 
\beq\label{R2>0}
	\left(n+1-\textstyle\frac{\gamma+1}{\gamma-1}\mu\right)^2>8k(V_8-V_4)(V_8-V_6).
\eeq
For each choice of $n=2$ or $3$ this inequality defines a certain region in the 
$(\lambda,\gamma)$-plane, which, according to our standing assumption
\eq{isntr_0<lam<1}${}_1$, is located within the half-strip $\{0<\lambda<1,\, \gamma>1\}$.
As the next lemma shows, the location of this region depends sensitively on $n$.
\begin{lemma}\label{R^2_pos}
	For $\lambda\in(0,1)$ fixed, we have
        \[\lim_{\gamma\uparrow+\infty}R^2(\lambda,\gamma)>0\qquad
        \Longleftrightarrow\qquad
        \left\{\begin{array}{ll}
        \lambda\in(0,\textstyle\frac{1}{9}) & \text{when $n=3$}\\\nn\\
        \lambda\in(\textstyle\frac{8}{9},1) & \text{when $n=2$.}\\
        \end{array}\right.\]
\end{lemma}
\begin{proof}
	Sending $\gamma\uparrow \infty$ in the expressions in \eq{aQ_2}, we get  
	\[a\to a_\infty:=\textstyle\frac{\mu}{m}-1<0,\qquad Q\to Q_\infty:=a_\infty^2.\]
	It follows from \eq{isntr_0<lam<1_V_pm} that $V_8=V_+\to \frac{1}{2}(a_\infty+|a_\infty|)=0$, 
	and $V_6=V_-\to \frac{1}{2}(a_\infty-|a_\infty|)=a_\infty$.
	Also, from \eq{k} and \eq{V_4} we have that $V_4\to0$, while $kV_4=-m(1+\mu)$. 
	As $\gamma\uparrow \infty$ the requirement
	 $R^2>0$ in \eq{R2>0} therefore reduces to the condition
	\beq\label{intermed1}
		(n+1-\mu)^2>8|a_\infty|\big[\lim_{\gamma\uparrow\infty} k(V_8-V_4)\big]
		=8(1-{\textstyle\frac{\mu}{m}})\big[m(1+\mu)+\lim_{\gamma\uparrow\infty} kV_8\big],
	\eeq
	where the last limit is of the form ``$\infty\cdot 0$.'' To analyze it we determine more precisely
	the distance between $Q$ and $Q_\infty$, and between $a$ and $a_\infty$, as 
	$\gamma\uparrow \infty$. Rewriting the expression \eq{aQ_2}${}_2$ for $Q$, we find that
	\[Q=(1-\textstyle\frac{\mu}{m})^2-\frac{4\mu}{m(\gamma-1)}
	\left[1+\frac{\mu}{m}\!\left(\frac{\gamma-2}{\gamma-1}\right)\right]
	\sim (1-\textstyle\frac{\mu}{m})^2-\frac{4\mu(m+\mu)}{m^2(\gamma-1)}
	\qquad\text{as $\gamma\uparrow\infty$.}\]
	To leading order in $\gamma$ we therefore have
	\[\sqrt Q\sim (1-\textstyle\frac{\mu}{m})-\frac{2\mu(m+\mu)}{m(m-\mu)}\frac{1}{(\gamma-1)}
	\qquad\text{as $\gamma\uparrow\infty$.}\]
	Combining this with 
	\[a=(\textstyle\frac{\mu}{m}-1)-\textstyle\frac{2\mu}{m}\frac{1}{(\gamma-1)},\]
	gives
	\[V_8=\textstyle\frac{1}{2}(a+\sqrt Q)\sim- \textstyle\frac{2\mu}{(m-\mu)}\frac{1}{(\gamma-1)}
	\qquad\text{as $\gamma\uparrow\infty$}.\]
	Recalling the expression \eq{k} for $k$ we obtain
	\[\lim_{\gamma\uparrow\infty} kV_8=-\textstyle\frac{mn\mu}{m-\mu}.\]
	Using this in \eq{intermed1} we conclude that, as $\gamma\uparrow\infty$, the requirement $R^2>0$ 
	reduces to the condition
	\beq\label{intermed2}
		(n+1-\mu)^2>8[(m-\mu)(1+\mu)-n\mu].
	\eeq
	Finally, with $n=3$ ($m=2$),  \eq{intermed2} becomes $8\mu+9\mu^2>0$, which reduces to 
	$\lambda=1+\mu<\frac 1 9$; with $n=2$ ($m=1$), \eq{intermed2} becomes $(1+\mu)(1+9\mu)>0$,
	which reduces to $\lambda>\frac 8 9$.
\end{proof}
Numerical plots of the curve in the $(\lambda,\gamma)$-plane defined by 
$R^2(\lambda,\gamma)=0$ reveal that it is the graph of an:
\begin{itemize}
	\item increasing function $\lambda\mapsto\gamma_3(\lambda)$ defined for 
	$\lambda\in (0,\frac{1}{9})$ and with a vertical asymptote at $\lambda=\frac{1}{9}$
	when $n=3$;
	\item decreasing function $\lambda\mapsto\gamma_2(\lambda)$ defined for 
	$\lambda\in(\frac{8}{9},1)$ and with a vertical asymptote at $\lambda=\frac{8}{9}$
	when $n=2$.
\end{itemize}
Figure \ref{gamma_3} shows the situation for $n=3$; the minimum value of 
$\gamma_3(\lambda)$ is $\gamma_*:=\gamma_3(0)\approx  8.72$. 

\begin{figure}
	\centering
	\includegraphics[width=6cm,height=6cm]{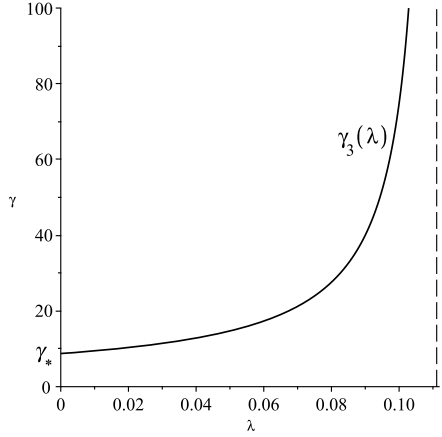}
	\caption{The graph of the function $\gamma_3(\lambda)$ which is defined 
	for $0<\lambda<\frac 1 9$;
	the discriminant $R^2(\lambda,\gamma)$ in \eq{R_sqrd}
	is positive when $\gamma>\gamma_3(\lambda)$; $\gamma_*=\gamma_3(0)\approx 8.72$. }\label{gamma_3}
\end{figure}

\subsubsection{The requirement $L_1<0$.}\label{L_1_neg}
For $n$ fixed we now consider $R^2$ and $L_1$ as functions of $\lambda$ and $\gamma$.
Recall that $L_1(\lambda,\gamma)<0$ is a necessary condition 
for having the trajectory $\Gamma_1$ connect $P_{+\infty}$
to $P_8$, and that we want $P_8$ to be a node, i.e., we need $R^2(\lambda,\gamma)>0$. 
The following lemma shows that only the case $n=3$ is favorable in this regard.
\begin{lemma}\label{L_1_neg_lemma}
	For $\lambda\in(0,1)$ and $\gamma>1$ we have
	\begin{enumerate}
		\item for $n=3$: if $R^2(\lambda,\gamma)>0$, then $L_1(\lambda,\gamma)<0$;
		\item for $n=2$: if $R^2(\lambda,\gamma)>0$, then $L_1(\lambda,\gamma)>0$.
	\end{enumerate}
\end{lemma}
\begin{proof}
	As detailed above, the requirement $L_1<0$ amounts to the inequality in \eq{L_1<0}.
	First, if $n=3$ and $R^2>0$ (so that $R>0$), then \eq{L_1<0} follows once we verify
	that $F_C<G_V$. A direct calculation shows that the latter inequality (for $n=3$)
	reduces to 
	\[4(\gamma-3)(\gamma-1)>\mu(3\gamma-5)(\gamma+1),\]
	which is trivially satisfied whenever $\gamma>3$ since $\mu=\lambda-1<0$. According to 
	the analysis above, $R^2>0$ implies $\gamma>\gamma_*>3$ when $n=3$, establishing 
	part (1) of the lemma.
	
	Next consider the case $n=2$. According to \eq{prmr_scdr_slopes_8} and 
	\eq{F_C_and_G_C_8}, the inequality $L_1>0$ amounts to $F_C-G_V>R$. 
	For $n=2$ it follows from \eq{F_C_8}, \eq{G_V_8}, and \eq{C_V_8_1},
	that $F_C>G_V$ if and only if
	\[2V_8+\textstyle\frac{\lambda}{2}>V_*,\]
	which is satisfied since $\lambda>0$ and $V_8>V_*$ (by Lemma \ref{V_locns}).
	If $R^2>0$, so that $R>0$, it follows that $F_C-G_V>R$ holds if and only if $(F_C-G_V)^2>R^2$.
	Substituting from \eq{R_sqrd} for $R^2$ shows that the latter inequality
	reduces to $F_VG_C<0$, which is satisfied according to \eq{F_C_and_G_C_8}${}_2$ and \eq{F_V_F_C}.
	We conclude that, for $n=2$, $L_1>0$ whenever $R^2>0$. 
\end{proof}
With this we have identified the relevant parameter regime in which to search for 
continuous similarity flows when $\lambda\in(0,1)$: we need to choose $n=3$,
and $(\lambda,\gamma)$ so that $R^2(\lambda,\gamma)>0$, i.e., 
$\lambda\in(0,\frac 1 9)$ and $\gamma>\gamma_3(\lambda)$. This guarantees
that $P_8$ is a node with a negative primary slope $L_1$. It remains to provide
examples in which the trajectory $\Gamma_1$ from $P_{+\infty}$ is drawn into $P_8$.
As noted above, a sufficient condition for this behavior is the existence of trajectories that enter
the region $\Omega$ (see \eq{Omega}) along its left edge at $V=V_*$, and from there 
continue on to reach $P_8$. The numerical verification of this condition is addressed 
in Section \ref{Gamma_num_verfcn}.

Before moving on to the behavior near $P_1$ we note the following 
consequence of the proof of Lemma \ref{L_1_neg_lemma}: 
when $n=3$ and $R^2(\lambda,\gamma)>0$, then 
$L_1(\lambda,\gamma)<L_2(\lambda,\gamma)<0$. Indeed, $L_1(\lambda,\gamma)
<L_2(\lambda,\gamma)$ holds by definition, while the inequality $L_2(\lambda,\gamma)<0$
amounts to $R<G_V-F_C$. The proof of Lemma \ref{L_1_neg_lemma} showed that 
when $n=3$ and $R^2(\lambda,\gamma)>0$, then $G_V-F_C>0$. Therefore,
$R<G_V-F_C$ holds provided $R^2<(G_V-F_C)^2$, which, according to 
\eq{R_sqrd}, amounts to $F_VG_C<0$. The latter inequality is satisfied according
to \eq{F_C_and_G_C_8}${}_2$ and \eq{F_V_F_C}.

\subsection{Behavior near $P_1$; construction of $\Gamma_2$}\label{P_1_analysis}
While there is (at most) a single trajectory $\Gamma_1$ joining $P_{+\infty}$ 
to $P_8$, there will be a continuum of trajectories joining $P_8$ to $P_9$
via the proper node $P_1$ at the origin.  
To identify these it is convenient to 
also classify the critical point $P_4$. Lazarus \cite{laz} shows that the 
the Wronskian there is given as 
\beq\label{W_4}
	W_4=2kC_4^2(V_4-V_6)(V_4-V_8),
\eeq
where $k$ is a positive constant (cf.\ \eq{W_8}). It follows from 
Lemma \ref{V_locns} that $W_4<0$ so that
$P_4$, and hence also $P_5$, are saddle points.
An inspection of the $(V,C)$-plane reveals the presence of three relevant separatrices
(see Figure \ref{Zoom_Gamma_2}):
\begin{itemize}
	\item $\Theta$ joining $P_8$ to $P_4$;
	\item $\Phi$ joining $P_4$ to $P_1$; and
	\item $\Psi$ joining $P_1$ to $P_5$.
\end{itemize}
Let $\zeta>0$ denote the slope of $\Psi$ at $P_1$; by symmetry, 
the slope of $\Phi$ at $P_1$ is then $-\zeta$.

The trajectories $\Gamma_2$ of interest to us (i.e., the ones leading to 
continuous Euler flows) are those that reach the origin $P_1$ 
from $P_8$, and then moves on to $P_9$ in the 
lower half-plane. 

As is clear from Figure \ref{Zoom_Gamma_2}, in order for $\Gamma_2$ 
to reach the origin, it must be located 
to the right of the separatrix $\Theta$. Also, it follows 
from the analysis at the end of Section \ref{L_1_neg} that all trajectories leaving 
$P_8$ do so with a negative slope: either $L_1$ or $L_2$, where $L_1<L_2<0$. 
Let $\Gamma_s$ 
denote the unique one among these which leaves with slope $L_2$.\footnote{If we 
use the trajectory $\Gamma_s$ the resulting Euler flow will suffer a {\em weak} discontinuity 
(i.e., a discontinuity in the 
first derivatives of the flow variables) across a curve 
$r(t)=(\frac{t}{x_8})^\frac{1}{\lambda}$ in the $(r,t)$-plane, where $x_8<0$ 
is such that $(V(x_8),C(x_8))=P_8$. This curve is 
a 1-characteristic for the radial Euler system \eq{m_eul}-\eq{ener_eul}.} 
Now, all trajectories leaving $P_8$ (with $V(x)$ increasing) proceed to cross 
vertically that part of the zero-level $\mathcal G=\{G=0\}$ which is 
located in the first quadrant. Among these, some reach the origin
with negative slopes after having vertically crossed also that part of $\mathcal G$ located in 
the second quadrant within the half-strip $\{(V,C)\,: V_4<V<0,\, C>0\}$. In all cases we have considered, it is clear
from numerical tests that there are other trajectories from $P_8$ which reach $P_1$ 
with positive slopes; see Remark \ref{no_way} below. In the following discussion it is 
assumed that this is the case. Evidently, the smallest positive slope with which $P_1$
can be reached from $P_8$ is that with which $\Gamma_s$ approaches $P_1$; 
we denote the latter slope by $\epsilon>0$. 

We next observe that the trajectory $\Gamma_2$ is not allowed to change 
its slope as it passes through $P_1$.\footnote{This points out a difference between $P_1$ and the critical points 
$P_6$-$P_9$. As noted above, a change of slope as $(V(x),C(x))$ passes through
$P_8$, say, results in a weak discontinuity in the corresponding Euler flow. 
In contrast, a change in slope at $P_1$ would generate, via \eq{alt_sim_vars} and \eq{well_bhvd},
an un-physical jump discontinuity across $t=0$.} 
Recalling the symmetry \eq{symms} of the phase portrait,
and setting $\delta:=\max(\zeta,\epsilon)>0$, we have that any  
of the infinitely many trajectories from $P_8$ which arrives at $P_1$ 
with a slope $s\in(-\infty,-\delta)\cup(\delta,+\infty]$, continues 
into the lower half-plane and reaches $P_9$. Indeed, the part of its 
trajectory in the lower half-plane will simply be the reflection about the $V$-axis
of one of the other trajectories from $P_8$ to $P_1$, viz.\ the one arriving at 
$P_1$ with slope $-s$. (In the limiting case
that $\Gamma_2$ reaches $P_1$ vertically, the lower part of $\Gamma_2$ is simply
the reflection of its upper part about the $V$-axis.) Any one of these trajectories may serve as
$\Gamma_2$ in our construction of continuous Euler flows.
\begin{remark}\label{no_way}
	If $\Gamma_s$ reached $P_1$ with slope $\epsilon<0$, then 
	none of the trajectories reaching the origin $P_1$ from $P_8$ would reach
	$P_9$. Again, we have not observed this scenario in any of our 
	numerical tests. E.g., in the case displayed in Figures \ref{Gamma_2}
	and \ref{Zoom_Gamma_2} (with $n=3$, $\lambda=0.02$, $\gamma=12$),
	$\epsilon$ is positive but so small that $\Gamma_s$ is indistinguishable 
	from the $V$-axis near the origin.
\end{remark}

\subsection{Summary}\label{sum}
We briefly summarize our findings so far in this section. First, by imposing the conditions  
$W>0$ and $R^2>0$ at the critical point $P_8$, we guarantee that 
$P_8$ is a node. The former requirement is automatically met once 
\eq{isntr_0<lam<1} holds, while the second requirement puts an $n$-dependent 
constraint on $\lambda$ and $\gamma$. The further condition 
$L_1(\lambda,\gamma)<0$, which is necessary in order that the trajectory $\Gamma_1$
connects to $P_8$ without first crossing $\mathcal L_+$, implies that the space 
dimension $n$ must be $3$. With $n=3$, the requirement $R^2>0$ is met, provided
$(\lambda,\gamma)$ lies above a certain graph $\gamma_3(\lambda)$ defined for 
$0<\lambda<\textstyle\frac 1 9$. If this is the case, then $L_1(\lambda,\gamma)<0$ is 
automatically satisfied.

Next, for all cases we have investigated numerically (see Section \ref{Gamma_num_verfcn}), 
there is an infinite number of trajectories $\Gamma_2$ joining $P_8$ to $P_9$ and passing 
through the origin $P_1$. Care must be taken that the trajectory from $P_8$ arrives at 
$P_1$ with a slope $s$ for which another trajectory arrives at $P_1$ from $P_8$ with the 
slope $-s$. When this holds, $\Gamma_2$ consists of the former trajectory together with 
the reflection of the latter about the $V$-axis.

Finally, the trajectory $\Gamma_3$ joining $P_9$ to $P_{-\infty}$ 
is the reflection of $\Gamma_1$ about the $V$-axis.

To finish the argument for the existence of continuous, radial similarity Euler flows, 
as described in the Main Results, it remains to verify 
the following two points: 
\begin{enumerate}
	\item[(i)] with $n=3$, there {\em are} choices of 
	$\lambda\in(0,\textstyle\frac 1 9)$ and $\gamma>\gamma_3(\lambda)$ 
	such that the trajectory $\Gamma_1$ from $P_{+\infty}$ reaches $P_8$, and 
	\item[(ii)] for these choices of the parameters there
	are trajectories $\Gamma_2$ from $P_8$ reaching the origin 
	and with the following property: $\Gamma_2$ reaches the origin with slope $s$,
	where $|s|>\delta$ ($\delta$ defined as above), and  
	there is another trajectory from $P_8$ reaching the origin with slope $-s$.
\end{enumerate}
Below we describe the numerical verification of these points.

\subsection{Numerical verification of (i) and (ii)}
\label{Gamma_num_verfcn}
The numerical verification is carried out with Maple.
As explained in Sections \ref{Gamma_1} and \ref{P_8_analysis}, to verify (i) it suffices to show that 
there is at least one solution crossing into the region $\Omega$ (see \eq{Omega}) along $V=V_*$ 
and reaching $P_8$ without first crossing the critical line $\mathcal L_+$.
To numerically check this, it is convenient to switch to the variables $W:=V-V_*$
and $Z:=C^{-2}$ which were used in Section \ref{P_infs} to analyze the 
critical point $P_{+\infty}$. The latter point is then located at the origin of the $(W,Z)$-plane. 
The analysis in Section \ref{P_infs} shows that this is a saddle point and that
$\Gamma_1$ leaves $P_{+\infty}=(0,0)$ with slope $\frac{dZ}{dW}=\frac{n+2}{B}=\frac 5 B$, where $B$ is
given in \eq{ab}. A sufficient condition for (i) to hold 
is that there are trajectories through points on the positive $Z$-axis which reach $P_8$
without first crossing $\mathcal L_+$.

Figure \ref{W_Z_plane} displays the situation in the $(W,Z)$-plane for $n=3$, 
$\lambda=0.02$, and $\gamma=12$, and provides clear numerical evidence that this is indeed 
the case. (The approximation of $\Gamma_1$, the dash-dot curve in Figure \ref{W_Z_plane}, 
is obtained by starting very
close to the origin along the straight line $Z= \frac 5 B W$.)

\begin{figure}
	\centering
	\includegraphics[width=6cm,height=6cm]{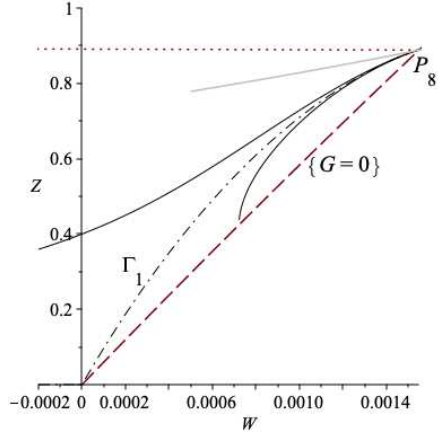}
	\caption{The dash-dot curve is an approximation of the trajectory $\Gamma_1$
	when plotted in the $(W,Z)$-plane, where $W=V-V_*$ and $Z=C^{-2}$. The dashed 
	curve is the zero-level $\mathcal G=\{G=0\}$ (crossed vertically by trajectories),
	and the dotted curve is the critical line $\mathcal L_+$. 
	The two solid curves are solutions: the upper one passes through the point $(W,Z)=(0,0.3)$
	while the lower one starts out near $\mathcal G$. Finally, the grey curve indicates the primary 
	slope with which the solutions reach $P_8$. The parameter values are $n=3$, $\lambda=0.02$, and
	$\gamma=12$. Evidently, $\Gamma_1$ reaches $P_8$ without first crossing 
	$\mathcal L_+$; this is confirmed by further numerical plots near $P_8$.}\label{W_Z_plane}
\end{figure}

\begin{figure}
	\centering
	\includegraphics[width=7cm,height=7cm]{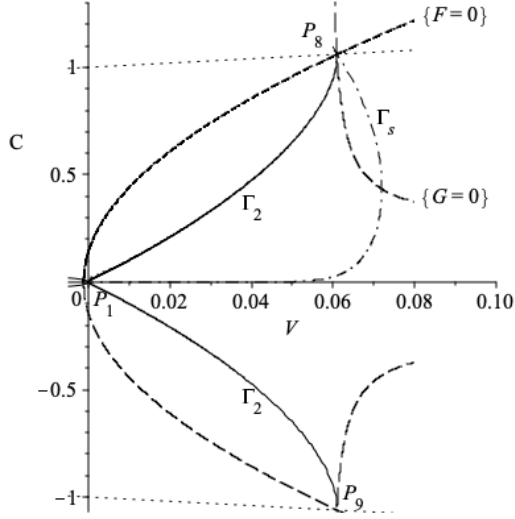}
	\caption{The dashed curves are the zero-levels $\{F=0\}$ and $\{G=0\}$, and 
	the dash-dot curve the $\Gamma_s$-trajectory which leaves $P_8$ 
	along the secondary direction.
	In addition there are two more solution trajectories: one is the separatrix $\Theta$ joining $P_4$ to $P_8$, 
	and the other is a complete $\Gamma_2$-trajectory (solid curve) joining $P_8$ and $P_9$ via the origin $P_1$.
	However, at this resolution the latter two trajectories are indistinguishable, and the $\Gamma_2$ trajectory 
	appears symmetric about the $V$-axis. (Figure \ref{Zoom_Gamma_2}
	displays a zoom-in near the origin which shows that this is not actually so.)
	The parameter values are as in Figure \ref{W_Z_plane}: 
	$n=3$, $\lambda=0.02$, and $\gamma=12$.}\label{Gamma_2}
\end{figure}

For numerical verification of (ii) we return to the $(V,C)$-plane and compute 
various trajectories joining the node at $P_8$ to the proper node (star point) $P_1$ at the origin. 
Figures \ref{Gamma_2}  and \ref{Zoom_Gamma_2} display the situation for the same parameter values as in 
Figure \ref{W_Z_plane}. Figure \ref{Gamma_2} shows the trajectory $\Gamma_s$ leaving $P_8$ along
the secondary direction and reaching the origin with a very small slope $\epsilon>0$. It also shows 
a complete $\Gamma_2$-trajectory, joining $P_8$ to $P_9$ via the origin, which appears to be 
symmetric about the $V$-axis. However, it is slightly 
un-symmetric and obscures the presence of the $P_4P_8$-separatrix $\Theta$;
see Figure \ref{Zoom_Gamma_2} for the detailed situation near the origin.

\begin{figure}
	\centering
	\includegraphics[width=7cm,height=7cm]{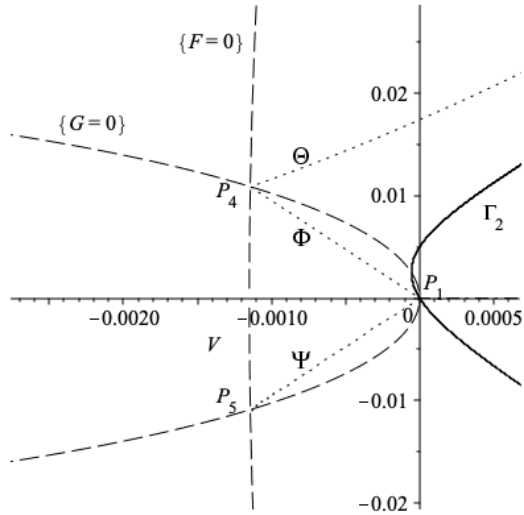}
	\caption{Zoom-in near the origin of Figure \ref{Gamma_2}. The 
	dashed curves are the zero-levels $\{F=0\}$ and $\{G=0\}$ the dotted curves
	are the separatrices $\Theta$, $\Phi$, and $\Psi$, and the solid curve is the 
	$\Gamma_2$-trajectory joining $P_8$ and $P_9$ via the origin $P_1$. Note that the 
	$\Gamma_s$-trajectory is indistinguishable from the $V$-axis 
	near the origin; thus $\epsilon\gtrsim0$ and $\delta=\zeta$ in this case.}\label{Zoom_Gamma_2}
\end{figure} 

\subsection{The flow at collapse and absence of shocks}\label{no_shock_graph}
Figures \ref{Gamma_2}-\ref{Zoom_Gamma_2} display a 
continuous trajectory $\Gamma_2$ joining the critical points $P_8$ and $P_9$
via the critical point $P_1$ at the origin. When this is joined with the trajectories 
$\Gamma_1$ and $\Gamma_3$ (the latter being the reflection of $\Gamma_1$ about 
the $V$-axis), we obtain a global, continuous solution $(V(x),C(x))$ 
of \eq{V_sim2}-\eq{C_sim2} joining $P_{+\infty}$ to $P_{-\infty}$ as $x$ varies from 
$-\infty$ to $+\infty$. Finally, from this we obtain, according to \eq{alt_sim_vars}, a 
globally defined, {\em continuous}, radial similarity Euler flow.

We now observe that whenever $\Gamma_2$ is a 
trajectory joining $P_8$ and $P_9$ via the origin, the same is true for its reflection $\Gamma_2'$ about the $V$-axis. 
Assuming $\Gamma_2$ reaches the origin with negative slope (as in Figures 
\ref{Gamma_2}-\ref{Zoom_Gamma_2}), the trajectory $\Gamma_2'$ 
will have positive slope at the origin. We therefore obtain two, physically distinct, 
Euler flows from these trajectories.
In particular, since both $\Gamma_2$ and $\Gamma_2'$ reach the origin as $x\uparrow0$, but 
with opposite signs of $V(x)$,
the  corresponding Euler flows display different behaviors at time of collapse:
in the one built from $\Gamma_2$ the fluid is moving toward  the origin at time of collapse, while
in the one built from $\Gamma_2'$ it moves outward. 

We find it noteworthy that the former flow remains continuous beyond collapse.
The flow variables $\rho(0,r)$, $u(0,r)$,  $c(0,r)$ all suffer gradient blowup at $r=0$ 
(see Section \ref{r=0_behavior}), and in addition the fluid flow is directed inward. 
It would be reasonable to expect that such data would generate 
an expanding shock wave for $t>0$. 

However, we can observe numerically that this is not what occurs for the flow corresponding to $\Gamma_2$.
In Figure \ref{Zoom_Hugoniot} we have plotted the trajectory $\Gamma_2$ from Figure \ref{Gamma_2},
together with the trajectory $\Gamma_3$, near $P_9$ (solid curve). We have also included the ``Hugoniot-locus'' 
$H$ of $\Gamma_2$ (dotted curve). This is the curve of points $(V_+,C_+)$ obtained from the 
Rankine-Hugoniot relations \eq{rh1}-\eq{rh2} as $(V_-,C_-)$ moves down from the origin along 
$\Gamma_2$. It may be deduced from the entropy condition that, for the solutions under 
consideration, a shock generated at collapse and propagating outward must necessarily be a 
3-shock which connects the outer state $(V_-,C_-)$, located above $\mathcal L_-$, to the 
inner state $(V_+,C_+)$ located below $\mathcal L_-$.
The presence of an expanding shock would then manifest itself by 
$H$ intersecting the trajectory $\Gamma_3$ at a point strictly 
below the critical line $\mathcal L_-$. However, Figure \ref{Zoom_Hugoniot} 
shows that $H$ (dotted curve) reaches $P_9$ without first
intersecting $\Gamma_3$: no shock is formed.

\begin{figure}
	\centering
	\includegraphics[width=6cm,height=6cm]{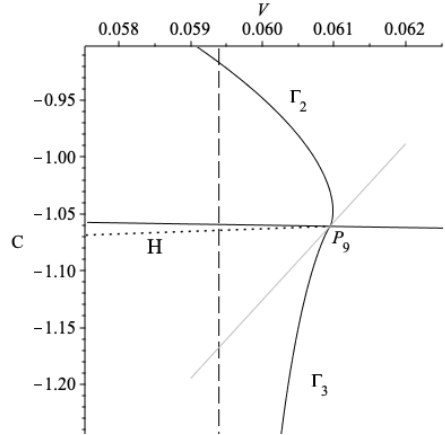}
	\caption{Zoom-in near the critical point $P_9$ with the same parameter values as 
	in Figures \ref{W_Z_plane}-\ref{Zoom_Gamma_2}: 
	$n=3$, $\lambda=0.02$, and $\gamma=12$. The solid straight 
	line is the critical line $\mathcal L_-=\{C=-1-V\}$, which appears almost horizontal. 
	The solid trajectory consists of parts of $\Gamma_2$ and $\Gamma_3$,
	located above and below $\mathcal L_-$, respectively; these meet at $P_9$
	along the primary direction (grey line). The dashed 
	vertical line is the asymptote $V=V_*$ approached by $\Gamma_3$. Finally, the 
	dotted curve is the Hugoniot locus $H$ consisting of states to which points along
	$\Gamma_2$ can jump. Note that $H$ is located below $\mathcal L_-$
	(as dictated by the entropy condition), but does not intersect $\Gamma_3$ 
	before reaching $P_9$. As a consequence, no shock wave occurs in the corresponding 
	Euler flow subsequent to collapse.} \label{Zoom_Hugoniot}
\end{figure}

\section*{Acknowledgements}
This material is based in part upon work supported by the National Science Foundation 
under Grant Number DMS-1813283 (Jenssen). Any opinions, findings, and conclusions 
or recommendations expressed in this material are those of the authors and do not 
necessarily reflect the views of the National Science Foundation.

The authors are grateful to Charis Tsikkou for help with the figures.


\end{document}